\documentclass[12pt]{article}
\usepackage[english]{babel}
\usepackage[normalem]{ulem}
\usepackage[all]{xy}
\usepackage{amsmath,amsthm,booktabs,color,dsfont,epsfig,graphicx,mathrsfs,multirow,scalefnt}
\usepackage{amsfonts, amssymb}
\usepackage{enumitem}

\newtheorem{theo}{Theorem}[section]
\newtheorem{cor}[theo]{Corollary}
\newtheorem{lem}[theo]{Lemma}

\newtheorem{prop}[theo]{Proposition}
\newtheorem{defn}[theo]{Definition}
\newtheorem{rmk}[theo]{Remark}
\newtheorem{ex}[theo]{Example}

\newcommand{\Z}{\mathbb{Z}}
\newcommand{\N}{\mathbb{N}}
\newcommand{\NZ}{\mathbb{S}}

\newcommand{\s}{\sigma}
\newcommand{\LA}{L_\mathbf{\Lambda}}
\newcommand{\LAF}{L_\mathbf{\bar{\Lambda}}}
\newcommand{\LAP}{L_{\underline{\mathbf{\Lambda}}}}
\newcommand{\LAH}{L_\mathbf{\hat{\Lambda}}}

\newcommand{\F}{\mathcal{F}}
\newcommand{\h}{\mathcal{H}}

\newcommand{\Pp}{\mathcal{P}}

\newcommand{\FF}{\mathfrak{F}}
\newcommand{\Ppp}{\mathfrak{P}}
\newcommand{\G}{\mathfrak{G}}

\newcommand{\SgHF}{\bar{\hat{\Lambda}}}
\newcommand{\SgHFH}{\hat{\bar{\hat{\Lambda}}}}
\newcommand{\SgHFHF}{\bar{\hat{\bar{\hat{\Lambda}}}}}
\newcommand{\SgHFHFH}{\hat{\bar{\hat{\bar{\hat{\Lambda}}}}}}

\title{Inverse semigroup shifts over countable alphabets}

\author{
\small{Daniel Gon\c{c}alves}\\
\footnotesize{UFSC -- Department of Mathematics}\\
\footnotesize{88040-900 Florian\'{o}polis - SC, Brazil}\\
\footnotesize{\texttt{daemi@mtm.ufsc.br}}
\and
\small{Marcelo Sobottka}\\
\footnotesize{UFSC -- Department of Mathematics}\\
\footnotesize{88040-900 Florian\'{o}polis - SC, Brazil}\\
\footnotesize{\texttt{sobottka@mtm.ufsc.br}}
\and
\small{Charles Starling}\\
\footnotesize{University of Ottawa}\\
\footnotesize{Department of Mathematics and Statistics}\\
\footnotesize{Ottawa, ON, K1J 1J9}\\
\footnotesize{\texttt{cstar050@uottawa.ca}}
}

\date{}

\begin{document}

\maketitle

\begin{abstract}
In this work we characterize shift spaces over infinite countable alphabets that can be endowed with an inverse semigroup operation.
We give sufficient conditions under which zero-dimensional inverse semigroups can be recoded as shift spaces whose correspondent inverse semigroup operation is a 1-block operation, that is, it arises from a group operation on the alphabet. Motivated by this, we go on to study block operations on shift spaces and, in the end, we prove our main theorem, which states that Markovian shift spaces, which can be endowed with a 1-block inverse semigroup operation, are conjugate to the product of a full shift with a fractal shift.
\end{abstract}

\bigskip
\hrule
\noindent
{\footnotesize\em This is a pre-copy-editing, author-produced PDF of an article accepted for publication in Semigroup Forum, following peer review.}
\hrule
\bigskip


\section{Introduction}

In this work we are concerned with semigroup operations defined on shift spaces over infinite countable alphabets.

Let $A$ be a non-empty countable set (called an alphabet) with the product topology. We define the sets $A^\N$, and $A^\Z$, as the sets of all one-sided infinite sequences over elements of $A$, and all two-sided infinite sequences over elements of $A$, respectively. That is,
$$A^\N:=\{(x_i)_{i\in\N}:\ x_i\in A\ \forall i\in\N\}$$
and
$$A^\Z:=\{(x_i)_{i\in\Z}:\ x_i\in A\ \forall i\in\Z\},$$
where $\N$ denotes the set of the non-negative integers, and $\Z$ denotes the set of the integers. Whenever a definition, or result, works for both $\N$ and $\Z$ we will use the symbol $\NZ$ meaning ``either $\N$ or $\Z$''. For now we consider $A^\NZ$ with the product topology. On $A^\NZ$ the map $\s:A^\NZ\to A^\NZ$, which shifts every entry of a given sequence one to the left, is called the {\em shift map}. A {\em shift space} is any subset $\Lambda$ of $A^\NZ$ which is closed and invariant under $\s$ (that is, $\s(\Lambda)=\Lambda$). Note that if the alphabet $A$ is a finite group then $A^\NZ$ becomes a topological group when given the operation of entrywise multiplication. In the two-sided case, the shift map is an expansive group automorphism on the zero dimensional group $A^\Z$.

A major inspiration for what follows is the seminal paper by Kitchens \cite{kitchens}, which concerns the converse of the above.

\begin{theo}{\em \cite[Theorem 1]{kitchens}}\label{Kmaintheo}
Suppose $X$ is a compact, zero dimensional topological group, and that $T:X\to X$ is an expansive group automorphism. Then
\begin{enumerate}
\item $(X, T)$ is topologically conjugate to $(Y,\sigma)$ via a group isomorphism, where $Y\subset A^\Z$ is a one-step shift of finite type and $A$ is a finite group;
\item $(X,T)$ is topologically conjugate to $(F, \tau)\times (B^\Z, \sigma)$ via a group isomorphism, where $F$ and $B$ are finite groups, $\tau$ is a group automorphism, and the group operation on $F\times B^\Z$ is given as an extension of $B^\Z$ by $F$.
\end{enumerate}
\end{theo}

A shift space in $A^\Z$ which carries a topological group operation which commutes with the shift is called a {\em topological group shift}, and the above shows that topological group shifts are modelled by group operations on the alphabet (after a possible re-coding). In \cite{sobottka2007}, similar results were obtained for topological quasigroup shifts over finite alphabets.

Here, we are interested in what one can say when the alphabet is countably infinite. In this case we will consider the compactification $\Sigma_A^\N$ of $A^\N$ proposed in \cite{Ott_et_Al2014}, which is such that the elements introduced when compactifying can be seen as finite words in $A$ together with an empty word $\O$. The two-sided analogue of this construction is the space $\Sigma_A^\Z$, which was considered in \cite{GoncalvesSobottkaStarling2015_2}.

If $A$ is a group, then producing a binary operation on $\Sigma_A^\N$ from the one on $A$ poses the immediate question of how to multiply sequences of different lengths. The natural thing to do is when multiplying two words $v$ and $w$ with the length of $v$ less than the length of $w$, that one would truncate $w$ to the length of $v$ and then multiply entrywise as normal. The result here of course is a word the same length as $w$. Under this operation, $\Sigma_A^\N$ is not a group -- indeed, $\O w = w\O = \O$ for every element $w\in \Sigma_A^\N$! It is however, an {\em inverse semigroup}, and it is for this reason that this paper concerns shift spaces with inverse semigroup operations on them -- ``inverse semigroup shifts'' -- and studies to what extent we can obtain results akin to Theorem \ref{Kmaintheo} in the infinite alphabet case (see Section \ref{sec: O-T-W shifts} for the details of the construction of $\Sigma_A^\NZ$).

After setting notation and background in Section 2, we prove a result similar to Theorem \ref{Kmaintheo}.i above for the one-sided case, Proposition \ref{section3maintheo}. To do this, we first define the notion of an {\em expansive partition}, and go on to show that if a dynamical system $(X,T)$ is an inverse semigroup and admits such a partition, then it is conjugate to an inverse semigroup shift. Section 4 is then dedicated to inverse semigroup shifts whose operation is induced from a group action on the alphabet. In Section \ref{sec_1-block operations} we see that a group operation over an infinite alphabet always induces a 1-block inverse semigroup operation on a full shift which uses that alphabet (but not necessarily a topological operation). Section \ref{algebraiccharacterization} is dedicated to showing that any inverse semigroup satisfying certain conditions is semigroup isomorphic to an inverse semigroup shift, while Section \ref{generalproperties1block} shows general properties of so-called 1-block operations. Here we show that, in contrast to the finite alphabet case, {\em follower sets} and {\em predecessor sets} may have different cardinalities. In Section \ref{final section}, we define a class of fractal shift spaces and then find an analogue to Theorem \ref{Kmaintheo}.ii, by showing that any two-sided Markovian 1-block inverse semigroup shift is isomorphic and topologically conjugate to the product of a fractal shift and a two-sided full shift, see Theorem \ref{lastsectionmaintheo}.

\section{Background}

In this section we will set notation and recall the background necessary for the paper. For detailed references on the topics covered in this section, see \cite{La98} for inverse semigroups, \cite{LindMarcus} and \cite{Ott_et_Al2014} for one sided shift spaces, \cite{GoncalvesSobottkaStarling2015_2} for two sided shift spaces over infinite alphabets, and \cite{GoncalvesSobottkaStarling2015} for sliding block codes between infinite alphabet shift spaces.

\subsection{Semigroups and dynamical systems}

Recall that a {\em semigroup} is a pair $(S,\cdot)$ where $S$ is a set and $\cdot$ is an associative binary operation. If $S$ is a topological space and $\cdot$ is continuous with respect to the topology of $S$, we say that $(S,\cdot)$ is a {\em topological semigroup}. An {\em inverse semigroup} is a semigroup $(S,\cdot)$ such that, for all $s\in S$, there exists a unique $s^*\in S$ such that $s\cdot s^*\cdot s = s$ and $s^*\cdot s \cdot s^* = s^*$.

We call an element $e\in S$ an {\em idempotent} if $e\cdot e = e$, and the set of all such elements will be denoted $E(S)$. It is true that, for all $s, t\in S$ and $e, f\in E(S)$ we have that $(s\cdot t)^* = t^*\cdot s^*$, $(s^*)^* =
s$, $e^* = e$, $e\cdot f = f\cdot e$ and $e\cdot f \in E(S)$. For all $s\in S$, the elements $s\cdot s^*$ and $s^*\cdot s$ are idempotents.

If $(S,\cdot)$ is an inverse semigroup with an identity (that is an element $1$ such that $1\cdot g=g\cdot 1=g$ for all $g\in S$) then it is called an {\em inverse monoid}. An element in $S$ is called a {\em zero} (and denoted by $0$) if $0\cdot g=g\cdot 0=0$ for all $g\in S$, and if $0,g, h \in S$, $g,h\neq 0$ and $g\cdot h=0$ then we say that $g$ and $h$ are {\em divisors of zero}. If they exist, the zero element and identity element are both seen to be unique.

Given an element $g\in S$, we have the Green's relations $\mathcal{L}$ and $\mathcal{R}$: $g\mathcal{L}h$ if and only if $g^*\cdot g=h^*\cdot h$; $g\mathcal{R}h$ if and only if $g\cdot g^*=h\cdot h^*$. Given $a\in S$ its {\em $\mathcal{L}$-class} is denoted $L(a):=\{b\in S:\
b^*\cdot b=a^*\cdot a\}$, and its {\em $\mathcal{R}$-class} is denoted $R(a):=\{b\in S:\ b\cdot b^*=a\cdot a^*\}$.

If $(S,\cdot)$ and $(\tilde S,\tilde \cdot)$ are two inverse semigroups, then a map $\xi:S\to \tilde S$ is said to be a {\em semigroup homomorphism} if for all $a,b\in S$ we have that $\xi(a\cdot b)=\xi(a)\tilde\cdot\xi(b)$. For such a $\xi$, for all $e\in E(S)$, we have that $\xi(e)\in E(\tilde S)$ and $\xi(a^*)=\xi(a)^*$ for all $a\in S$. If a semigroup homomorphism $\xi$ is bijective then we say that it is an {\em isomorphism}.

 Any inverse semigroup $(S,\cdot)$ carries a natural partial order $\leq$ given by saying that $a\leq b$ if and only if there exists some $e\in E(S)$ such that $a = e\cdot b$.

A {\em dynamical system (DS)} is a pair $(X,T)$, where $X$ is a locally compact space and $T:X\to X$ is a map, and if in addition $T$ is continuous then we will say that $(X,T)$ is a {\em topological dynamical system (TDS)}. Two dynamical systems $(X,T)$ and $(Y,S)$ are said to be {\em conjugate} if there exists a bijective map $\Phi:X\to Y$ such that $\Phi\circ T=S\circ \Phi$. If, in addition, $(X,T)$ and $(Y,S)$ are topological dynamical systems and $\Phi$ is a homeomorphism we say
that $(X,T)$ and $(Y,S)$ are {\em topologically conjugate} and the map $\Phi$ is called a {\em (topological) conjugacy} for the dynamical systems. 

 We say that a dynamical system $(X,T)$ is {\em expansive} if there exists a partition $\mathcal{U}$ of $X$ such that, for all $x\neq y$ in $X$, there exists $n \in\NZ$ such that $T^n(x)$ and $T^n(y)$ are in different elements of $\mathcal{U}$ (where $\NZ=\N$ if $T$ is not invertible, and $\NZ=\Z$ otherwise). In such a case, $\mathcal{U}$ is called an {\em expansive partition for $T$}.

Lastly, recall that a locally compact Hausdorff space is said to be {\em zero dimensional} or {\em totally disconnected} if it has a basis consisting of clopen sets.

\subsection{Compactified shift spaces over countable alphabets}\label{sec: O-T-W shifts}

Let $A$ be a countable set with $|A|\geq 2$, which we will call an {\em alphabet} and whose elements will be called {\em symbols} or {\em letters}. Define a new symbol $\o\notin A$, which we will call the {\em empty letter} and define the extended alphabet $\tilde{A}:=A\cup\{\o\}$. We will consider the set
of all infinite sequences over $\tilde{A}$, $\tilde{A}^\NZ:=\{(x_i)_{i\in\N}: x_i\in \tilde{A}\ \forall i\in\NZ\}$
and define $\Sigma_A^{\NZ\ \text{inf}},\Sigma_A^{\NZ\ \text{fin}}$ by
$$\Sigma_A^{\NZ\ \text{inf}}:=A^\NZ:=\{(x_i)_{i\in\NZ}:\ x_i\in A\ \text{ for all } i\in \NZ\}$$ and $$\Sigma_A^{\NZ\ \text{fin}}:=\{(x_i)_{i\in\NZ}:\ x_i \in \tilde{A}, x_j = \o\text{ for some }j, \text{ and if } x_i=\o\ \text{then} \ x_{i+1}=\o, \forall i\in\NZ\}.$$
We will denote by $\O$ the constant sequence whose entries are all the empty letter $\o$, that is, $\O=(x_i)_{i\in\NZ}$ where $x_i=\o$ for all $i\in\NZ$.

\begin{defn}[One-sided full shift] The {\em one-sided full shift} over $A$ is the set $$\Sigma_A^\N:=\left\{\begin{array}{lcl}\Sigma_A^{\N\ \text{inf}} & \text{, if} & |A|<\infty\\\\ \Sigma_A^{\N\ \text{inf}}\cup\Sigma_A^{\N\ \text{fin}} & \text{, if} &
|A|=\infty\end{array}\right..$$

\end{defn}

\begin{defn}[Two-sided full shift] The {\em two-sided full shift} over $A$ is the set $$\Sigma_A^\Z:=\left\{\begin{array}{lcl}\Sigma_A^{\Z\ \text{inf}}\cup\{\O\} & \text{, if} & |A|<\infty\\\\ \Sigma_A^{\Z\ \text{inf}}\cup\Sigma_A^{\Z\ \text{fin}} & \text{, if} &
|A|=\infty\end{array}\right..$$

\end{defn}

\begin{defn} For $x\in\Sigma_A^\NZ$, we define  $$l(x):=\left\{\begin{array}{lcl}+\infty&,\ if& x\in \Sigma_A^{\NZ\ \text{inf}}\\\\
                                                                                            \displaystyle \max_{k\in\NZ }\{k: x_k\neq\o\} &,\ if& x\in \Sigma_A^{\NZ\ \text{fin}}\setminus\{\O\}\\\\
                                                                                            -\infty &,\ if& x=\O
                                                                                            \end{array}\right.$$
and call $l(x)$ the {\em length} of $x$.
\end{defn}

We will refer to sequences in $\Sigma_A^{\NZ\ \text{fin}}$ as {\em finite sequences}, sequences in $\Sigma_A^{\NZ\ \text{inf}}$ as {\em infinite sequences} and we will refer to $\O$ as the {\em empty sequence}. To simplify the notation, giving a sequence $(x_i)_{i\in\NZ}\in\Sigma_A^{\NZ\ \text{fin}}\setminus\{\O\}$ we will frequently omit the entries with the empty letter when denoting it, that is, we will denote $(x_i)_{i\leq k}=(\ldots x_1x_2x_3\ldots x_k)$, where $k=l(x)$.

Given $(x_i)_{i\leq k}\in\Sigma_A^{\NZ\ \text{fin}}\setminus\{\O\}$ and a finite set $F\subset A$, we define
\begin{equation}\label{cylinder1}Z(x,F):=\{y\in\Sigma_A^\NZ:\ y_i=x_i\ \forall i\leq k,\ y_{k+1}\notin F\}.\end{equation}
In the special case that $F=\emptyset$, then we denote $Z(x,F)$ simply by $Z(x)$. Furthermore, given $x^1,x^2,\ldots,x^m\in\Sigma_A^{\NZ\ \text{fin}}\setminus\{\O\}$
we define
\begin{equation}\label{cylinder2}Z^c(x^1,\ldots,x^m)= \Big(\bigcup_{j=1}^m Z(x^j)\Big)^c . \end{equation}
In the case $\NZ=\N$, given $F=\{x^1,\ldots,x^m\}$, where each $x^i$ is a finite word of length 1, we define $Z(\O,F)=Z^c(x^1,\ldots,x^m)$.

The sets of the form \eqref{cylinder1} and \eqref{cylinder2} are called {\em generalized cylinders} and they form the basis of the topology that we consider in $\Sigma_A^\NZ$. For this topology we have that $\Sigma_A^\NZ$ is zero dimensional (generalized cylinders are clopen sets), Hausdorff and compact. Furthermore, when $\NZ=\N$ the topology is metrizable (see \cite{Ott_et_Al2014}, Section 2), while for $\NZ=\Z$ the topology is not first countable (there is no countable neighborhood basis for $\O$ - see \cite[Proposition 2.7]{GoncalvesSobottkaStarling2015_2}). We refer the reader to \cite{GoncalvesSobottkaStarling2015_2} and \cite{Ott_et_Al2014} for more details about this topology.

Note that if $\NZ=\N$ then $Z(x)\cap \Sigma_A^\N$ coincides with an
usual cylinder of the product topology in $\Sigma_A^{\N\ \text{inf}}$. In particular if $\NZ=\N$, and we allow $A$ to be finite, then $\Sigma_A^\N$ with the topology generated by the generalized cylinders is a classical shift space over $A$ (see \cite[Remark 2.24]{Ott_et_Al2014}).

As mentioned in the introduction, the shift map $\s:\Sigma_A^\NZ\to\Sigma_A^\NZ$ is defined as the map given by $\s((x_i)_{i\in\NZ})=(x_{i+1})_{i\in\NZ}$ for all $(x_i)_{i\in\NZ}\in\Sigma_A^\NZ$. Note that
$\s(\O)=\O$ and, for all $x\in\Sigma_A^{\NZ\ \text{fin}}$, $x\neq \O$, we have that $l\big(\s(x)\big)=l(x)-1$. Furthermore, when $A$ is infinite and $\NZ=\N$ we have that $\s:\Sigma_A^\N\to\Sigma_A^\N$ is continuous everywhere but at $\O$. In any other situation (either under $\NZ=\N$ and finite $A$, or under $\NZ=\Z$ and countable $A$) we have that $\s:\Sigma_A^\NZ\to\Sigma_A^\NZ$ is a continuous map (see \cite[Proposition 2.12]{GoncalvesSobottkaStarling2015_2}).

Given $\Lambda\subseteq\Sigma_A^\NZ$, let $\Lambda^{\text{fin}}:=\Lambda\cap\Sigma_A^{\text{fin}}$ and $\Lambda^{\text{inf}}:=\Lambda\cap\Sigma_A^{\text{inf}}$ be
the set of finite sequences and the set of infinite sequences in $\Lambda$, respectively.
We consider on $\Lambda$ the topology induced from $\Sigma_A^\NZ$, that is, the topology whose basis are the sets of the form  $Z_\Lambda(x,F):=Z(x,F)\cap\Lambda$ and $Z_\Lambda(x^1,...,x^m):=Z(x^1,...,x^m)\cap\Lambda$.

For each $n\geq 1$, let
$$B_n(\Lambda):=\{(a_1\ldots a_n):\ \text{there exists } x\in\Lambda,\ i\in\NZ, \text{ such that } x_{i+j-1}=a_j\text{ for all } j=1,\ldots,n\}$$
be the set of all {\em blocks of length $n$} in $\Lambda$.

 Furthermore, define $$B_{\text{linf}}(\Lambda):=\left\{\begin{array}{ll}\emptyset &\text{ if } \Lambda\subset\Sigma_A^\N\\\\
                                                                          \left\{(x_i)_{i\leq k}:\  x\in \Lambda,\  k\in\Z\right\}&\text{ if } \Lambda\subset\Sigma_A^\Z.\end{array}\right.$$
The {\em language} of $\Lambda$ is $$B(\Lambda):=B_{\text{linf}}(\Lambda)\cup\bigcup_{n\geq 1}B_n(\Lambda).$$

Also, define the {\em letters} of $\Lambda$ to be $\LA:= B_1(\Lambda)\setminus\{\o\}$ -- this is the set of elements of $A$ used in sequences of $\Lambda$.

Given $a\in B(\Sigma_A^\NZ)$, the {\em $k^{th}$ follower set} of $a$ in $\Lambda$ is the set \begin{equation}\label{followersets}\F_k(\Lambda,a):=\{b\in B_k(\Lambda):\ ab\in B(\Lambda)\},\end{equation}while the {\em $k^{th}$ predecessor set} of $a$ in $\Lambda$ is\begin{equation}\label{predecessorsets}\Pp_k(\Lambda,a):=\{b\in B_k(\Lambda):\ ba\in B(\Lambda)\}.\end{equation}\\

We say that a subset $\Lambda$ satisfies the {\em infinite extension property} if, for all $x\in\Sigma_A^{\NZ\ \text{fin}}\setminus\{\O\}$,  $$x\in\Lambda^{\text{fin}}\qquad\Longleftrightarrow\qquad\Big|\F_1\big(\Lambda,(x_i)_{i\leq\l(x)}\big)\Big|=\infty.$$
Notice that when $\NZ=\N$, $\O\in\Lambda^{\text{fin}}$ if, and only if, $|A|=\infty$, while if $\NZ=\Z$ then we have that $\O\in\Lambda^{\text{fin}}$ if, and only if, $|\Lambda|=\infty$.

\begin{defn}\label{2-sided_shifts} A set $\Lambda\subseteq\Sigma_A^\NZ$ is said to be a {\em shift space over $A$} if the following three properties hold:
\begin{enumerate}
\item[1.] $\Lambda$ is closed with respect to the topology of $\Sigma^\NZ_A$;
\item[2.]$\Lambda$ is invariant under the shift map, that is, $\s(\Lambda) = \Lambda$;
\item[3.] $\Lambda$ satisfies the infinite extension property.
\end{enumerate}
\end{defn}

We notice that condition 2 could be relaxed when $\NZ=\N$ to $\s(\Lambda) \subset \Lambda$. However in this work we will just consider, even when $\NZ=\N$, shift spaces such that $\s(\Lambda) = \Lambda$. We remark that $\Lambda$ satisfies the infinite extension property if and only if $\Lambda^{\text{inf}}$ is dense in $\Lambda$ (see \cite[Proposition 3.8]{Ott_et_Al2014} and \cite[Lemma 2.18]{GoncalvesSobottkaStarling2015_2}).

 Given $X\subset\Sigma_A^{\NZ\ \text{inf}}$ such that $\s(X)=X$ we can define the shift space generated by $X$ as the smallest shift space $\Lambda$ such that $\Lambda^{\text{inf}}\supset X$. It follows that $\Lambda:=cl(X)$, where $cl(X)$ stands for the closure of $X$.

Shift spaces can be characterized in terms of forbidden words:

\begin{defn} Let $\mathbf{F}$ be a subset of $B(\Sigma_A^\NZ)$ whose elements do not use the empty letter. We define $X_{\mathbf{F}}$ as the subset of $\Sigma_A^\NZ$ such that: (i) its infinite sequences are exactly those that do not contain a block belonging to $\mathbf{F}$; (ii) its finite sequences satisfy the `infinite extension property'; (iii) $\s(X_{\mathbf{F}})=X_{\mathbf{F}}$ (where this last assumption is only necessary to assure that, when $\NZ=\N$, condition 2 of Definition \ref{2-sided_shifts} is satisfied).
\end{defn}

From \cite[Theorem 3.16]{Ott_et_Al2014} and \cite[Proposition 2.25]{GoncalvesSobottkaStarling2015_2} we have that $\Lambda$ is a shift space if, and only if, $\Lambda= X_{\mathbf{F}}$ for some $\mathbf{F}\subset B(\Sigma_{\LA}^\NZ)$ whose elements do not use the empty letter.

\begin{defn} Given a shift space $\Lambda\subset\Sigma_A^\Z$, the projection of $\Lambda$ onto the
non-negative coordinates is the map $\pi:\Lambda\to\Sigma_A^\N$ given by $\pi\big((x_i)_{i\in\Z}\big)= (x_i)_{i\in\N}$.
\end{defn}

The projection $\pi$ establishes a relationship between two-sided shift spaces and one-sided shift spaces. We refer the reader to sections 2 and 4 of \cite{GoncalvesSobottkaStarling2015_2} for more details about it.

We now summarize some important types of shift spaces. If $\Gamma$ is a shift space, then we say that $\Gamma$ is a:

\begin{itemize}

\item \uline{{\sc shift of finite type (SFT)}}: if $\Lambda=X_\mathbf{F}$ for some finite $\mathbf{F}\subset B(\Sigma_{\LA}^\NZ)$;

\item \uline{{\sc finite-step shift}}: if $\Lambda=X_\mathbf{F}$ for some $\mathbf{F}\subset B(\Lambda)\setminus B_{\text{linf}}(\Lambda)$. In the case that $\mathbf{F}\subseteq B_{M+1}(\Lambda)$, we will say that $\Lambda$ is an $M$-step shift ($1$-step shifts will also be referred as {\em Markov} (or {\em Markovian}) shifts);

\item \uline{{\sc infinite-step shift}}: if it is not a finite-step shift. Note that just two-sided shift spaces can contain proper infinite-step shifts;

\item \uline{{\sc edge shift}}: if $\Lambda(G):=X_\mathbf{F}$, where $\mathbf{F}:=\{ef:\ t(e)\neq i(f)\}$ and $G=(E,V,i,t)$ is a directed graph with no sources and sinks.

\item \uline{{\sc row-finite shift}}: if for all $a\in\LA$ we have that $\F_1(\Lambda,a)$ is a finite set;

\item \uline{{\sc column-finite shift}}: if for all $a\in\LA$ we have that $\Pp_1(\Lambda,a)$ is a finite set.
\end{itemize}

\subsection{Sliding block codes}\label{Sliding block codes}

Roughly speaking, a slinding block code is a code which encodes a sequence $(x_i)_{i\in\NZ}$ of a shift space as a new sequence (possibly in other shift space) in a continuous way and which is invariant by translations.

A {\em pseudo cylinder} of a shift space $\Lambda\subset\Sigma_A^\NZ$ is a set of the form
$$[b]_{k}^\ell:=\{(x_i)_{i\in\Z}\in \Lambda: (x_{k}\ldots x_\ell)=b\},$$
where $k\leq\ell\in\NZ$ and $b=(b_1\ldots b_{-k+\ell+1})\in B(\Sigma_A^\NZ)$.
We say that the pseudo cylinder $[b]_{k}^\ell$ has memory  $K:=-\min \{0,k\}$ and anticipation $L:=\max\{0,\ell\}$. Note that if $\Lambda$ is a one-sided shift space, then any pseudo-cylinder of $\Lambda$ has memory equal to zero. We adopt the convention that the empty set is a pseudo cylinder of $\Lambda$ whose memory and anticipation are zero.

Given $C\subset \Lambda\subset \Sigma_A^\NZ$, we will say that $C$ is {\em finitely defined in $\Lambda$} if there exist two collections of pseudo cylinders of $\Lambda$, namely $\{[b^i]_{k_i}^{\ell_i}\}_{i\in I}$ and $\{[d^j]_{m_j}^{n_j}\}_{j\in J}$,  such that
\begin{equation*}
C=\bigcup_{i\in I} [b^i]_{k_i}^{\ell_i}\qquad \text{and}\qquad
C^c= \bigcup_{j\in J}[d^j]_{m_j}^{n_j}.
\end{equation*}
In this case we say that $C$ has memory  $K:=-\inf_{i\in I} \{0,\ k_i\}$ and anticipation $L:=\sup_{i\in I}\{0,\ \ell_i\}$.

\begin{defn}\label{defn_sliding block code}
Let $A$ and $B$ be two countably infinite alphabets, and let $\Lambda \subset \Sigma_A^\NZ$ and $\Gamma \subset \Sigma_B^\NZ$ be two shift spaces. Suppose $\{C_a\}_{a\in \tilde{B}}$ is a pairwise
disjoint partition of $\Lambda$, such that:
\begin{description}\addtolength{\itemsep}{-0.5\baselineskip}

\item[\em 1.] for each $a\in \tilde{B}$ the set $C_a$ is finitely defined in $\Lambda$;

\item[\em 2.] $C_{\o}$ is shift invariant (that is, $\s(C_{\o})\subset C_{\o}$).

\end{description}

A map $\Phi:\Lambda\to\Gamma$ is called a {\em sliding block code} if

\begin{equation*}\label{LR_block_code}\bigl(\Phi(x)\bigr)_n=\sum_{a\in \tilde{B}}a\mathbf{1}_{C_a}\circ\sigma^{n}(x),\quad \forall x\in\Lambda,\ \forall n\in\NZ, \end{equation*} where $\mathbf{1}_{C_a}$ is the
characteristic function of the set $C_a$ and $\sum$ stands for the symbolic sum.

 In this case, we say that $\Phi$ has memory $K:=\sup_{a\in \tilde{B}} \{k_a\}$ and anticipation $L:=\sup_{a\in \tilde{B}} \{\ell_a\}$, where $k_a$ and $\ell_a$ stand for the memory and the anticipation of each $C_a$, respectively,
 If $K,L<\infty$ we will say that $\Phi$ is a $K+L+1$-block, while if $K=\infty$ or $L=\infty$, we will say that $\Phi$ has unbounded memory or anticipation, respectively.
\end{defn}

Intuitively, $\Phi:\Lambda\subset\Sigma_A^\NZ\to \Gamma\subset\Sigma_B^\NZ$ is a sliding block code if it has a local rule, that is, if for all $n\in\Z$ and $x\in\Lambda$, there exist $k,\ell\geq 0$ which depends on the configuration of $x$ around $x_n$, such that we only need to know $(x_{n-k}x_{n-k+1} \ldots x_{n+\ell})$ to determine $\bigl(\Phi(x)\bigr)_n$. In the particular case that $\Phi$ is a sliding block code with memory $K$ and anticipation $L$, then its local rule can be written as a function $\alpha:B_{K+L+1}(\Lambda)\to\tilde{B}$ such that for all $x\in\Lambda$ and $n\in\NZ$, it follows that $\bigl(\Phi(x)\bigr)_n=\alpha(x_{n-K}x_{n-K+1} \ldots x_{n+L})$.

We refer the reader to \cite{GoncalvesSobottkaStarling2015},  \cite{GoncalvesSobottkaStarling2015_2} and \cite{SG} for more details about sliding block codes between shift spaces over infinitely countable alphabets.

\subsection{Shift spaces as dynamical systems and semigroups}

If $\Lambda\subset\Sigma_A^\NZ$ is a shift space, then $(\Lambda,\s)$ is a dynamical system. If $\NZ=\N$, \cite{GoncalvesSobottkaStarling2015} ensures that $(\Lambda,\s)$ is a topological dynamical system if, and only if, $\Lambda$ is a column-finite shift. On the other hand, if $\NZ=\Z$, then \cite{GoncalvesSobottkaStarling2015_2} ensures that $(\Lambda,\s)$ is always a topological dynamical system.

\begin{defn}\label{shift_conjugate}
Two shift spaces $\Lambda\subseteq\Sigma^\NZ_A$ and $\Gamma\subseteq\Sigma^\NZ_B$ are said to be (topologically) conjugate if the dynamical systems $(\Lambda,\s)$ and $(\Gamma,\s)$ are (topologically) conjugate.
\end{defn}

We remark that a sliding block code is always a conjugacy between shift spaces (see \cite[Proposition 3.12]{GoncalvesSobottkaStarling2015} and \cite[Proposition 3.10]{GoncalvesSobottkaStarling2015_2}). Furthermore, in \cite{GoncalvesSobottkaStarling2015} and \cite{GoncalvesSobottkaStarling2015_2} conditions are given for a sliding block code to be a topological conjugacy.

We now seek to incorporate a semigroup operation to the shift space structures defined so far.

\begin{defn} We will say that a shift space $\Lambda\subseteq\Sigma_A^\NZ$ is a {\em semigroup shift} if there exists a binary operation $\bullet$ defined on $\Lambda$ such that $(\Lambda,\bullet)$ is a semigroup and the
shift map is a semigroup homomorphism. If, in addition, the operation $\bullet$ is continuous, then we will say that $(\Lambda,\bullet)$ is a {\em topological semigroup shift}.
\end{defn}

A particular type of semigroup shifts that we will study are the {\em $k$-block semigroup shifts}.

\begin{defn}\label{defn_block_oper} Let $A$ be an alphabet, $\Lambda \subset \Sigma_A^\NZ$ be a subshift, and $\bullet$ be a binary operation on $\Lambda$. Suppose that $\{C_a\}_{a\in \tilde{A}}$ a pairwise disjoint partition of $\Lambda\times\Lambda$, such that:
\begin{description}\addtolength{\itemsep}{-0.5\baselineskip}

\item[\em 1.] for each $a\in \tilde{A}$ the set $C_a$ is finitely defined;

\item[\em 2.] $C_{\o}$ is shift invariant.

\end{description}

We say that $\bullet$ is a {\em block operation} if

\begin{equation}\label{eq_block_oper_1}\bigl(x\bullet y\bigr)_n=\sum_{a\in \tilde{A}}a\mathbf{1}_{C_a}\circ\big(\sigma^n(x),\sigma^n(y)\big),\quad \forall x,y\in\Lambda,\ \forall n\in\NZ,\end{equation} where
$\mathbf{1}_{C_a}$ is the characteristic function of the set $C_a$, and $\sum$ stands for the symbolic sum.

Denoting by $k_a$ and $\ell_a$ the memory and the anticipation of each $C_a$, respectively, let $K:=\inf_{a\in \tilde{A}}k_a$ and  $L:=\sup_{a\in \tilde{A}}\ell_a$. If $K$ and $L$ are finite, we will say that $\bullet$ is a
$K+L+1$-block operation with memory $K$ and anticipation $L$ (in the case $\NZ=\N$ we always have $K=0$).
\end{defn}

Once again it is best to think of a block operation as an operation with a local rule. In the particular case of $\bullet$ being a block operation with memory $K$ and anticipation $L$, its local rule can be written as $\alpha: B_{K+L+1}(\Lambda\times\Lambda)\to \tilde{A}$, such that given $x, y\in\Lambda$ and $n\in\NZ$,
\begin{equation*}\label{eq_block_oper_2}
(x\bullet y)_n = \alpha(x_{n-K}\cdots x_{n+L}\ ,\ y_{n-K}\cdots y_{n+L}).
\end{equation*}

One notices that the above definition of a block operation is similar to the definition of sliding block codes (see \cite{GoncalvesSobottkaStarling2015,GoncalvesSobottkaStarling2015_2, SG}). However, since $\Lambda\times\Lambda$ is not a shift space, the map $(x,y)\mapsto x\bullet y$ is not a sliding block code and we cannot use the results established for them. At any rate, block operations share several properties with sliding block codes, and, for the particular case of 1-block semigroups (see Subsection \ref{sec_1-block operations}), such an operation will be a sliding block code defined from the shift space $\Lambda\boxtimes\Lambda$ to $\Lambda$ (see Section \ref{Product of Shift Spaces} for the definition of $\Lambda\boxtimes\Lambda$).

\subsection{The Product of Shift Spaces}\label{Product of Shift Spaces}


Given two shift spaces $\Lambda\subseteq\Sigma_A^\NZ$ and $\Gamma\subseteq\Sigma_B^\NZ$, the Cartesian product $\Lambda\times\Gamma$ is not a shift space on the alphabet $A\times B$ (if either $A$ or $B$ is infinite). In
light of this, we define product of shift spaces as follows:

\begin{defn} Given two shift spaces $\Lambda\subseteq\Sigma_A^\NZ$ and $\Gamma\subseteq\Sigma_B^\NZ$, the product of $\Lambda$ and $\Gamma$ is the shift space $\Lambda\boxtimes\Gamma\subseteq \Sigma_{A\times B}^\NZ$
generated by $\Lambda^{\text{inf}}\times\Gamma^{\text{inf}}$, that is, $\Lambda\boxtimes\Gamma:=\overline{\Lambda^{\text{inf}}\times\Gamma^{\text{inf}}}$, where we are identifying an element $(a_i)_{i\in \NZ} \times (b_i)_{i \in \NZ}  \in \Lambda^{\text{inf}}\times\Gamma^{\text{inf}}$ with $\left( (a_i,b_i) \right)_{i \in \NZ} \in \Sigma_{A\times B}^\NZ$ .
\end{defn}

We remark that $\Lambda\boxtimes\Gamma$ is the smallest shift space whose set of infinite sequences is $\Lambda^{\text{inf}}\times\Gamma^{\text{inf}}$. Alternatively, we can consider the map
$\wp:\Lambda\times\Gamma\to\Sigma_{A\times B}^\NZ$, defined, for all $(a,b)\in\Lambda\times\Gamma$ and $i\in\NZ$, by:
\begin{equation}\label{map_wp}
\big(\wp(a,b)\big)_i:=\left\{
\begin{array}{lcl}
\big(a_i,b_i) &\text{if}& a_i,b_i\neq\o\\\\
\o &\text{if}& a_i=\o\text{ or }b_i=\o
\end{array}\right.
\end{equation}
and it follows that $\Lambda\boxtimes\Gamma=\wp(\Lambda\times\Gamma)$, $(\Lambda\boxtimes\Gamma)^{\text{inf}}=\wp(\Lambda^{\text{inf}}\times\Gamma^{\text{inf}})$, and $(\Lambda\boxtimes\Gamma)^{\text{fin}}=\wp(\Lambda\times\Gamma\setminus\Lambda^{\text{inf}}\times\Gamma^{\text{inf}})$.

\begin{prop}\label{wp_continuity}
Let $\Lambda\subset\Sigma_A^\NZ$ and $\Gamma\subset\Sigma_B^\NZ$ be  two shift spaces, consider $\Lambda\times\Gamma$ with the product topology and $\Lambda\boxtimes\Gamma$ with the topology generated by generalized cylinders. Then the map
$\wp:\Lambda\times\Gamma\to\Lambda\boxtimes\Gamma$ defined by \eqref{map_wp} is continuous.

\end{prop}

\begin{proof}

Suppose that $\Lambda\subseteq\Sigma_A$ and $\Gamma\subseteq\Sigma_B$ are two shift spaces.
Let $$\mathcal{Z}:=Z_{\Lambda\boxtimes\Gamma}\big(\kappa,F\big),$$ be a generalized cylinder of $\Lambda\boxtimes\Gamma$, where $\kappa=(v_i,w_i)_{i\leq k}\in \Sigma_{A\times B}^{\NZ\ \text{fin}}\setminus\{\O\}$, $k=l(\kappa)$ and $F\subset A\times B$\sloppy\ is a finite set, say $F=\big\{(f_1,g_1),\ldots,(f_m,g_m)\big\}$.

Let $F':=(f_i)_{1\leq i\leq m}\in A$ and $F'':=(g_i)_{1\leq i\leq m}\in B$, where $(f_i,g_i)\in F$ for all $i$ (notice that $F'$ and $F''$ are finite sequences and not sets). It follows that if $\wp^{-1}(\mathcal{Z})\neq\emptyset$, then

\small
$$\begin{array}{lcl}\wp^{-1}(\mathcal{Z})&=& \displaystyle \Big\{\big(\ldots(x_1,y_1)\ldots(x_k,y_k)(x_{k+1},y_{k+1})\ldots\big)\in\Lambda\times\Gamma:\ (x_i,y_i)=(v_i,w_i)\ \forall i\leq k,\ (x_{k+1},y_{k+1})\notin
F\Big\}\\\\
&=& \displaystyle \big[Z_{\Lambda\times\Gamma}(v)\times Z_{\Lambda\times\Gamma}(w)\big]\cap\bigcap_{j=1}^m\Big[Z_{\Lambda\times\Gamma}(v^j)\times
Z_{\Lambda\times\Gamma}(w^j)\Big]^c,
\end{array}
$$
\normalsize
where $v=(v_i)_{i\leq k}$, $w=(w_i)_{i\leq k}$, $v^j=(v^j_i)_{i\leq k+1}$ with $v^j_i = v_i$ for $i\leq k$ and $v^j_{k+1}=f_j$ and $w^j=(w^j_i)_{i\leq k+1}$ with $w^j_i = w_i$ for $i\leq k$ and $w_{k+1}=g_j$.

Since for each $j=1,\ldots, m$, the set $\big[Z_{\Lambda\times\Gamma}(v^j)\times
Z_{\Lambda\times\Gamma}(w^j)\Big]$ is a clopen of $\Lambda\times\Gamma$, then its complement
 is also a clopen of $\Lambda\times\Gamma$. Hence $\wp^{-1}(\mathcal{Z})$ is a finite intersection of clopen sets
and thus it is a clopen set of $\Lambda\times\Gamma$.

On the other hand, given $\kappa^1,\kappa^2,\ldots,\kappa^m\in\Sigma_{A\times B}^{\NZ\ \text{fin}}\setminus\{\O\}$
we have that
$$\wp^{-1}\Big(Z^c(\kappa^1,\ldots,\kappa^m)\Big)= \wp^{-1}\left(\Big(\bigcup_{j=1}^m Z(\kappa^j)\Big)^c\right)=
\bigcap_{j=1}^m \Big(\wp^{-1}\big(Z(\kappa^j)\big)\Big)^c.$$
By the first part of the proof we have that each $\wp^{-1}\big(Z(\kappa^j)\big)$ is a clopen set and hence we conclude that
$\wp^{-1}\Big(Z^c(\kappa^1,\ldots,\kappa^m)\Big)$ is also a clopen set.
\end{proof}


\section{Expansive dynamics and inverse semigroup shifts}


In \cite{kitchens}, it is proven that if $X$ is a compact zero-dimensional group and $T$ is an expansive and transitive group automorphism of $X$, then $(X,T)$ is conjugate to the full shift over a finite group. In our situation, the full shift over a countably infinite group, when given the operation of entrywise multiplication, has the structure of an inverse semigroup with the empty sequence being the zero element. The shift map fixes the empty sequence, and in general the empty sequence is the only possible point of discontinuity of the shift map.

In this section we prove a result similar to Kitchens', that is under some conditions, an expansive automorphism over a zero-dimensional inverse semigroup can be modeled by the shift map over $\Sigma_A^\N$, for a suitable infinite alphabet $A$.

\begin{lem}\label{inverseopen}
Let $X$ be a compact space and let $T:X\to X$ be a function. Suppose that there exisits $p\in X$ such that $T(p) = p$ and that $T$ is continuous on $X\setminus \{p\}$. Then, for all $U\subset X\setminus \{p\}$
and all $k>0$, the set $T^{-k}(U)$ is open.
\end{lem}
\begin{proof}
Take $U\subset X\setminus \{p\}$ to be open. Note that $T^{-k}(U)$ will not contain $p$ for any $k$, because $T^k(p) = p$. Hence we only need to show that $T^{-1}(U)$ is open.

If $x\in T^{-1}(U)$, then $T(x) \in U$. Since $T$ is continuous at $x$, there exists an open set $V$ with $x\in V$ such that $T(V) \subset U$, that is, $V \subset T^{-1}(T(V)) \subset T^{-1}(U)$. Hence, $T^{-1}(U)$ is
open.

\end{proof}

\begin{prop}\label{shiftconj}
Let $X$ be a compact Hausdorff space, and let $T:X\to X$ be a function. Suppose that we have $p\in X$ such that $T(p) = p$ and that $T$ is continuous on $X\setminus \{p\}$. Further, suppose that there exists $\{ U_i\}_{i\in A}$ an infinite
partition of clopen sets of $X\setminus \{p\}$, such that $\mathcal{U} = \{ U_i\}_{i\in I} \cup \{\{p\}\}$ is an expansive partition for $T$. Then there exists an infinite alphabet $A$ and a continuous injection $h: X \to \Sigma_A^\N$, which is a homeomorphism onto
its image, such that $h\circ T = \sigma \circ h$. If we assume further that
\begin{equation}\label{denseset}
X\setminus \bigcup_{n\geq 0}T^{-n}(p)
\end{equation}
is dense in $X$, then the image of $h$ is a subshift of $\Sigma_A^\N$.
\end{prop}

\begin{proof}
Let $\{ U_i\}_{i\in A}$ be a clopen partition of $X\setminus \{p\}$, and let $\mathcal{U} = \{ U_i\}_{i\in I} \cup \{\{p\}\}$ be an expansive partition for $T$. Consider the full shift $\Sigma_A^\N$ over the index set $A$.

Let $h: X\to \Sigma_A^\N$ be defined by
\[
h(p) = \O
\]
\[
h(x) = \begin{cases} (a_i)_{i = 0}^{\infty} & \text{if } T^i(x) \in U_{a_i}\text{ and } T^i(x) \neq p \text{ for all } i\\ (a_i)_{i = 0}^{K}, & \text{if }T^i(x) \in U_{a_i}\text{ and } T^K(x) \neq p \text{ and } T^{K+1} =
p.\end{cases}
\]

Since $\{ U_i\}_{i\in A}$ is an expansive partition, it is easy to see that $h$ is injective. We claim that $h$ is a homeomorphism onto its image, and that its image is a subshift of $\Sigma_A^\N$.

We first prove that $h$ is continuous. Suppose that $x_n \to x$ in $X$. We have to prove that $h(x_n) \to h(x)$ in $\Sigma_A^\N$. If $l(h(x)) = \infty$, then $h(x)$ is an infinite sequence in $A$. Take $M\in\N$ and let
\[
V = \bigcap_{j = 0}^M T^{-j}(U_{h(x)_j}).
\]
Note that $V$ is open (by Lemma \ref{inverseopen}) and nonempty, and that $p\notin V$. Since $x_n\to x$, there exists $N$ such that for all $n\geq N$, $x_n\in V$. Thus, for all $n\geq N$, the first $M$ entries of the
sequence $h(x_n)$ are equal to that of $h(x)$. Hence, $h(x_n) \to h(x)$.

Suppose now that $l(h(x)) = K$ with $0 < K < \infty$. Let $F\subset A$ be finite, say $F = \{f_1, \dots, f_l\}$. We have to find $N\in\N$ such that $n\geq N$ implies that $l(h(x_n)) \geq l(h(x))$, $h(x_n)_i = h(x)_i$ for
$1\leq i \leq K$. and $h(x_n)_{K+1} \notin F$.

Let
\begin{equation}\label{FU}
F_\mathcal{U} = \bigcup_{i = 0}^l U_{f_i}
\end{equation}
\[
V = \left(\bigcap_{j = 0}^K T^{-j}(U_{h(x)_j})\right)\bigcap T^{-(K+1)}\left(F_\mathcal{U}^c\right).
\]
We note that $x\in V$, since $T^{K+1}(x) = p$. In addition, we have that $p\notin V$. Hence, $T^{K+1}$ is continuous at all points in $V$. We also note that $T^K(V)$ does not contain $p$, because it is contained in
$U_{h(x)_K}$. Hence, $T^{K+1}$ is continuous on $V$.

We claim that $V$ is open. Let $y\in V$. Then $T^{K+1}(y)\in F_\mathcal{U}^c$, an open set. Since $T^{K+1}$ is continuous at $y$, there exists an open set $W_y$ containing $y$ such that $T^{K+1}(W_y)\subset
F_{\mathcal{U}}^c$, that is, $W_y\subset T^{-(K+1)}(F_\mathcal{U}^c)$. Hence
\[
W_y \cap \bigcap_{j=0}^K T^{-j}(U_{h(x)_j})
\]
is open, is contained in $V$, and contains $y$. Hence, $V$ is open.

Now, since $x_n \to x$, there exists $N >0$ such that whenever $n\geq N$, $x_n \in V$. This implies that whenever $n \geq N$, we have $l(h(x_n)) \geq l(h(x))$, $h(x_n)_i = h(x)_i$ for $0\leq i \leq K$, and $h(x_n)_{K+1}$.
Hence we have proven that $h(x_n)$ converges to $h(x)$.

Finally, suppose $l(h(x)) = 0$. Then $x=p$. Let $F\subset A$ be finite, say $F = \{f_1, \dots, f_l\}$, and take $F_\mathcal{U}$ as in \eqref{FU}. Since $x_n \to x$, and $F_\mathcal{U}^c$ is an open set containing $x$, there
exists $N>0$ such that for all $n\geq N$ we have $x_n \in F_\mathcal{U}^c$. This means that for all $n\geq N$, $h(x_n)_1\notin F$. Since $F$ was arbitrary, this means that $h(x_n)\to \O = h(x)$.

Now that we have that $h$ is continuous and injective, since $X$ is compact and $\Sigma_A^\N$ is Hausdorff, $h$ is automatically a homeomorphism onto its image.

We now prove our final statement. We have that $h(X)$ is closed and is clearly closed under the shift, and so it remains to show that it has the infinite extension property. Suppose that $y\in h(X)$ and that $l(y) = K$ with
$0 < K < \infty$. Then $y = h(x)$ with $T^{K+1}(x) = p$. As before, let $F\subset A$ be finite, form $F_\mathcal{U}$, and let
\[
V = \bigcap_{j = 0}^K T^{-j}(y_j) \bigcap T^{-(K+1)}(F_\mathcal{U}^c),
\]
which is an open neighborhood of $x$. Again, $T^{K+1}$ is continuous on $V$. Because the set in \eqref{denseset} is dense in $X$, there exists $z\in X\setminus \bigcup_{n\geq 0}T^{-n}(p)$ in $V$. Now, the first $K$ entries
of $h(z)$ match $y$. Furthermore, the element $h(z)_{K+1}\in A$ must be different from any element of $F$ because of the form of $V$. Since $F$ was arbitrary, we are forced to conclude that the set
\[
\{ a\in A \mid y\gamma \in h(X), \gamma_1 = a\}
\]
is infinite. Hence, $h(X)$ has the infinite extension property.

\end{proof}

\begin{rmk}From now on we suppose that we have a compact Hausdorff space $X$, a map $T: X\to X$ which satisfies the hypotheses of Proposition \ref{shiftconj} and denote the subshift of $\Sigma_A^\N$ given by the
image of the map $h$ by $\Lambda$.
\end{rmk}

\begin{cor}If $T^{-1}(p) = \{ p \}$ then $\Lambda$ is a row-finite subshift of $\Sigma_A^\N$  .
\end{cor}

\begin{proof}

It is clear that if $T^{-1}(p) =  \{ p \}$ then $T^{-n}(p) =  \{ p \}$ for all $n = 0, 1, 2, \ldots$ and so the only finite word in $\Lambda$ is the zero word and hence, by Proposition 3.21 in \cite{Ott_et_Al2014}, the
result follows.

\end{proof}

\begin{prop}\label{section3maintheo}
Suppose that $X$ and $T$ are as in the assumptions of Proposition \ref{shiftconj} and let $\Lambda:=h(X)$ be the shift space topologically conjugate to $(X,T)$, which is given by Proposition \ref{shiftconj}. Suppose further that there exists an inverse semigroup operation on $X$, such that $p\cdot x = x\cdot p = p$ for all $x \in X$, that $T$ is an inverse
semigroup homomorphism and that, for all $i, j \in A$, there exists $k \in A$ such that $U_i \cdot U_j \subseteq U_k$. Then there exists a 1-block inverse semigroup operation $\bullet$ on $\Lambda$, such that $(X, \cdot)$ is isomorphic to  $(\Lambda,\bullet )$.
\end{prop}

\begin{proof}

First notice that the alphabet $A $ inherits the multiplication operation from the multiplication on $U_i$, that is, if $U_{a_i}\cdot U_{b_i}\subseteq U_{c_i}$ then $ a_i \cdot b_i = c_i$.

Let $h: X \to \Lambda$ be as in Proposition \ref{shiftconj} and let multiplication in $\Lambda$, $\bullet: \Lambda \times \Lambda \to \Lambda$,
be defined by $$ x \bullet y := h \left(h^{-1}(x) \cdot h^{-1}(y) \right),$$
for all $x = (x_n)$ and $y = (y_n)$ in $\Lambda$.

From the definition of $\bullet$ it is clear that $(X,\cdot)$ is isomorphic to $(\Lambda, \bullet)$. All we need to show is that $\bullet$ is a 1-block operation. So, let $z_1= (a_i)$ and $z_2 = (b_i)$ be in
$\Lambda$. Notice that $z_1 = h(x)$ and $z_2 = h(y)$ for some $x,y \in X$. We now divide the proof in three cases:

First, suppose that both $z_1$ and $z_2$ have infinite length. Then $z_1 \bullet z_2 = h(x \cdot y) = (c_i)$, where $c_i$ is such that $T^i(x \cdot y) = T^i(x) \cdot T^i(y) $ belongs to $U_{c_i}$. Since $T^i(x) \in U_{a_i}$
and $T^i(y) \in U_{b_i}$ it follows that $c_i = a_i \cdot b_i$.

Now, suppose that $z_1$ has finite length and $z_2$ has infinite length. Let $K$ be the smallest integer such that $T^K(x)= p$.  Then, since $T^i(x) \in U_{a_i}$ and $T^i(y) \in U_{b_i}$ for all $1\leq i \ < K$, and $ p
\cdot U_{b_i} = p$ for all $i \geq K$, we have that $z_1 \bullet z_2 = h(x \cdot y) = (c_i)_{i=1}^{K-1}$, where $c_i$ is such that $T^i(x \cdot y) = T^i(x) \cdot T^i(y) $ belongs to $U_{c_i}$ for all $i< K$ (what as before
implies that $c_i = a_i \cdot b_i$).

Finally the case when both $z_1$ and $z_2$ have finite length is analogous to the above and we leave it to the reader.

\end{proof}

\section{Block operations for shift spaces}\label{block operations}


In this section we will present some general properties of block operations defined on shift spaces over countable alphabets. We shall pay particular attention to 1-block operations.

\begin{prop} If $\Lambda\subset\Sigma_A^\NZ$ is a shift space and $\bullet$ is a block operation in $\Lambda$, then the shift map is a homomorphism for $\bullet$.
\end{prop}

\begin{proof}
The proof is similar to that in \cite[Proposition 3.12]{GoncalvesSobottkaStarling2015}.
\end{proof}

The following result is analogous to Corollaries 3.13-3.15 in \cite{GoncalvesSobottkaStarling2015}.

\begin{cor}\label{standard_facts}
If $\bullet$ is a block operation in a shift space $\Lambda$, then:

\begin{enumerate}

\item if $x,y\in\Lambda$ are sequences with period $p,q\geq 1$, respectively (that is, $\s^p(x)=x$ and $\s^q(y)=y$), then $x\bullet y$ has period $r:=LCM(p,q)$;

\item $\O\bullet\O$ is a constant sequence;

\item if $\O\bullet\O=\O$ then the product of two finite sequences of $\Lambda$ is also a finite sequence of $\Lambda$.
\end{enumerate}
\end{cor}

\qed

\subsection{1-block operations}\label{sec_1-block operations}

We now focus on characterizing subshifts with a 1-block operation for which the empty word $\O$ is a zero.

\begin{defn}\label{canonical_oper} Let $\cdot$ be a binary operation on $A$. We extend it to $\tilde{A}$ by defining $a\cdot\o=\o\cdot a:=\o$ for all $a\in A$, $\o \cdot \o = \o$, and define the canonical 1-block operation $\bullet$, induced by $\cdot$, on the shift space $\Sigma_A^\NZ$ as
follows:

\begin{equation}\label{inducedOperation}(x_i)_{i\in\NZ}\bullet(y_i)_{i\in\NZ}:=(x_i \cdot y_i)_{i\in\NZ},\qquad \text{ for all } (x_i)_{i\in\NZ},(y_i)_{i\in\NZ}\in\Sigma_A^\NZ.\end{equation}

\end{defn}

Note that the shift map is a homomorphism for $\bullet$ (since it is a block operation) and $\O$ is a zero for $\bullet$, that is, $x\bullet\O=\O\bullet x=\O$ for all $x\in\Sigma_A^\NZ$.

The next proposition gives a sufficient condition under which a 1-block operation on a shift space $\Lambda$ is induced from a operation on the alphabet as in \eqref{inducedOperation}.

\begin{prop}\label{1-block_alph} Let $\Lambda\subseteq\Sigma_A^\NZ$ be a shift space which is endowed with a 1-block operation $\bullet$. Suppose that $\bullet$ has no zero divisors and $\O\bullet x=x\bullet\O=\O$ for all
$x\in\Lambda$. Then there exists a binary operation $\cdot$ on the alphabet of $\Lambda$, which canonically induces $\bullet$.
\end{prop}

\begin{proof}
Suppose that $\bullet$ is a 1-block binary operation on $\Lambda$. This means that there exists $\alpha: \LA\cup\{\o\}\times \LA\cup\{\o\} \to \LA\cup\{\o\}$, such that for all $x, y\in\Lambda$ and $n\in\NZ$,
$$(x\bullet y)_n = \alpha(x_n,y_n).$$
Hence, since $\Lambda$ has no zero divisors, defining on $\LA$ the binary operation $\cdot$ given (for $a,b\in \LA$) by $$a\cdot b:= \alpha(a,b),$$ we get a binary operation on the alphabet which induces
$\bullet$.

\end{proof}

Although the above proposition gives a one-to-one correspondence between operations on the alphabet $A$ and a class of operations on the shift $\Sigma_A$, in general the induced operation on $\Sigma_A^\NZ$ and the corresponding
operation on $A$ do not have the same properties (Proposition \ref{group-inversesemigroup} gives an example of this).

Next we characterize continuous binary operations on $\Sigma_A^\N$ (where $\Sigma_A^\N\times\Sigma_A^\N$ is given the product topology of $\Sigma_A^\N$) in terms of the operation on the alphabet.

\begin{theo} \label{continuityfinite} Let $\Lambda\subseteq \Sigma_A^\N$ be shift space and $\bullet$ a 1-block multiplication induced by an operation $\cdot$ on $\LA$. Then $\bullet$ is a topological operation
if and only if for all $a\in \LA$ the set $\{ (b,c)\in \LA \times \LA: b\cdot c =a \}$
is finite.
\end{theo}

\begin{proof} Let $\Phi:\Lambda\times\Lambda\to\Lambda$ be given by $\Phi(x,y):= x\bullet y$. Then $\Phi$ induces a map $\hat\Phi:\Lambda\boxtimes\Lambda\to\Lambda$ defined for all $(a_i,b_i)_{i\in\N}\in\Lambda\boxtimes\Lambda$ by and $n\in\N$ by
$$\Big(\hat\Phi\big((a_i,b_i)_{i\in\N}\big)\Big)_n:=\left\{\begin{array}{lcl}
a_i\cdot b_i &,\ if& i\leq l\big((a_i,b_i)_{i\in\N}\big)\\\\
\o &,\ if& i> l\big((a_i,b_i)_{i\in\N}\big)
\end{array}\right.
$$

Notice that we can now write $\Phi=\hat \Phi\circ\wp$. Since Proposition \ref{wp_continuity} assures that $\wp$ is continuous, we just need to check that $\hat\Phi$ is continuous. However, $\hat\Phi$ is a sliding block code (see Definition 3.7 in \cite{GoncalvesSobottkaStarling2015}) and, furthermore, $\hat\Phi^{-1}(\O)=\{\O\}$. Therefore, from Theorem 3.16 of \cite{GoncalvesSobottkaStarling2015}, it follows that $\hat\Phi$ is continuous if, and only if, for each $a\in \LA$ the set $\hat\Phi^{-1}\big(Z_\Lambda(a)\big)=\{ (b,c)\in \LA \times \LA: b\cdot c =a \}$ is finite.

\end{proof}

We end this section with an example which will be important in the sequel.

\begin{ex}\label{example_top_semi_group} Let $A=\bigcup_{i^\in\N}G_i$, where $\{(G_i,\cdot_i)\}_{i\in\N}$ is a collection of disjoint finite groups. Define for $a,b\in A$ the operation
$$a\cdot b:=\left\{\begin{array}{ll} a\cdot_i b &\text{ if } a,b\in G_i \\
a &\text{ if }  a\in G_i, b\in G_j \text{ and } i>j\\
b &\text{ if }  a\in G_i, b\in G_j \text{ and } i<j
\end{array}\right. .$$

Then $(\Sigma_A^\N,\bullet)$, induced as in Definition \ref{canonical_oper}, is a topological inverse semigroup.
\end{ex}

\subsection{1-block inverse semigroup shifts induced from group operations}
\label{algebraiccharacterization}
In this section whenever $G$ is a group we will denote by $1_G$ its identity. Below we characterize the 1-block operations arising from a group multiplication on the alphabet. We begin by summarizing the properties of the shift space that we know so far, and then prove that these properties characterize such shifts in the one-sided case.

\begin{prop}\label{continuityGroup}
Suppose that $(G,\cdot)$ is an infinite countable group and let $\bullet$ be the operation on $\Sigma_G^\NZ$ induced from $\cdot$. Then the following are true:
\begin{enumerate}\addtolength{\itemsep}{-0.5\baselineskip}

\item $(\Sigma_G^\NZ,\bullet)$ is an inverse monoid with zero;

\item The sequence $e^\infty:=(\ldots 1_G1_G1_G\ldots)$ is the identity of $(\Sigma_G^\NZ,\bullet)$;

\item The empty sequence $e^{-\infty}:=\O$ is the zero of $(\Sigma_G^\NZ,\bullet)$;

\item $(\Sigma_G^\NZ,\bullet)$ has no divisors of zero;

\item $E(\Sigma_G^\NZ)$ (the set of idempotents in $\Sigma_G^\NZ$) is a linearly ordered set:
\[
E(\Sigma_G^\NZ) = \{e^{-\infty}\leqslant \ldots \leqslant e^n \leqslant e^{n+1} \leqslant \ldots \leqslant e^\infty:\ n\in\NZ\}
\]
where for all $n\in\NZ$ the sequence $e^n$ is such that $l(e^n)=n$ and $e^n_i=1_G$ for all $i\leq n$;

\item For each $n\in\NZ\cup\{\pm\infty\}$ the idempotent $e^n\in E(\Sigma_G^\NZ)$ is such that its $\mathcal{R}$-class and $\mathcal{L}$-class coincide with the set $\Sigma_{G\ n}^\NZ:=\{x\in\Sigma_G^\NZ:\ l(x)=n\}$. Hence $(\Sigma_{G\ n}^\NZ,\bullet)$ is a topological group;

\item For all $e\in E(\Sigma_G^\NZ)$ and $x\in \Sigma_G^\NZ$, we have that $e\bullet x = x\bullet e$;

\item Denote by $\Sigma_{G\ 1}^\NZ$ the set of all words with length 1 in $\Sigma_G^\NZ$. If $\NZ=\N$, then $(\Sigma_{G\ 1}^\NZ,\bullet)$ is isomorphic to $(G,\cdot)$. If $\NZ=\Z$, then $\big(\pi(\Sigma_{G\ 1}^\NZ),\bullet\big)$ is isomorphic to $(G,\cdot)$, where $\pi$ is the projection on the non-negative coordinates;

\item The shift $\sigma$ is a surjective semigroup homomorphism such that $\sigma(e^{-\infty})=e^{-\infty}$, $\sigma(e^\infty)=e^\infty$, and  $\sigma(e^n) = e^{n-1}$ for all $n\in\NZ$ (with the convention that if $\NZ=\N$, then $e^0:=e^{-\infty}$);

\item If $x\in \Sigma_G^\NZ$ is such that $e\bullet x = e$ for some $e\in E(\Sigma_G^\NZ)\setminus\{\O\}$ and $\sigma(x)$ is an idempotent, then $x$ is also an idempotent and $e\leq x$;

\item The sequence $(e^n)_{n\in\N}$ converges to $e^\infty$.
\end{enumerate}
\end{prop}

\begin{proof}
All the above results follow directly from Definition \ref{canonical_oper}.
\end{proof}

We now show that the properties presented in Proposition \ref{continuityGroup} are sufficient conditions over an inverse semigroup so it can be embedded in a one-sided full shift over a infinite countable group alphabet.

\begin{prop}\label{chains1}
Suppose that $(S,\cdot)$ is an inverse monoid with 0. Suppose further that

\begin{description}\addtolength{\itemsep}{-0.5\baselineskip}
\item[\em 1.] $S$ has no zero divisors;
\item[\em 2.] The set of idempotents is a countable linearly ordered set
\[
E(S)=\{0 = e_0 \leqslant e_1 \leqslant e_2 \leqslant \cdots \}\cup\{e_\infty=1\},
\] where, as usual, $e_i \leq e_j$ means $e_i\cdot e_j = e_i$, and $e_\infty$ is the unique upper bound of $\{e_i\}_{i\geq 0}$ in $S$;
\item[\em 3.] There exists a surjective homomorphism $T: S \to S$ such that $T(e_i) = e_{i-1}$, for $1\leq i< \infty$;
\item[\em 4.] For all $s\in S$, $e_1\cdot s = s \cdot e_1$;
\item[\em 5.] If $s\in S$ is such that $T(s)$ is an idempotent and $e_1\cdot s = e_1$, then $s$ is also an idempotent.

\end{description}
Let $\hat S:= S\setminus\{0\}$. Then $(e_1\cdot\hat S,\ \cdot)$ is a group and the map $\theta: S \to \Sigma_{e_1\cdot\hat S}^\N$, defined by
\[
\theta(s) = \begin{cases}(e_1\cdot s, e_1\cdot T(s), e_1\cdot T^2(s), \dots) &\text{if } T^i(s) \neq e_0 \text{ for all } i\\
(e_1\cdot s, e_1\cdot T(s), \dots, e_1\cdot T^i(s),\o,\o,\ldots) &\text{if } T^{i}(s) \neq e_0 \text{ and } T^{i+1}(s) = e_0\\
\O & \text{if }  s = e_0
\end{cases},
\]
is an injective semigroup homomorphism between $(S,\cdot)$ and $(\Sigma_{e_1\cdot\hat S}^\N,\bullet)$ such that $\sigma\circ\theta = \theta\circ T$, where $\bullet$ is induced from the operation $\cdot$ on $e_1\cdot\hat
S$.
\end{prop}

\begin{proof}
We show that since $T$ is a surjective homomorphism, we have $T(e_0)=e_0$ and
$T(e_\infty)=e_\infty$. Take $a\in S$ and find $b\in S$ such that $T(b)=a$, and calculate $T(e_0)\cdot a=T(e_0)\cdot T(b)=T(e_0\cdot b)=T(e_0)$ and $T(e_\infty)\cdot a=T(e_\infty)\cdot T(b)=T(e_\infty\cdot b)=T(b)=a$. Since
the zero and identity are unique in an inverse semigroup, $T(e_0)=e_0$ and $T(e_\infty)=e_\infty$.

Denote $\Lambda:=\theta(S)\subseteq \Sigma_{e_1\cdot \hat S}^\N$. This set is invariant under the shift map and we have that $\LA=e_1\cdot \hat S$ and $\cdot$ is the operation on $\LA$ which induces
$\bullet$.

Let us show that $e_1\cdot \hat S$ is a group (under the operation inherited from $S$). For any $e_1\cdot s \in e_1\cdot \hat S$,
$$(e_1\cdot s)\cdot (e_1\cdot s)^* = e_1\cdot s\cdot s^*\cdot e_1 =e_1,$$
 where the second equality follows from the fact that $s\cdot s^*\neq e_0$ (due to hypothesis 1) and that $e_1$ is the smallest nonzero idempotent (hypothesis 2). A similar computation holds for $(e_1\cdot s)^*(e_1\cdot
 s)$. Furthermore, it is straightforward to check that $e_1$ is the identity for elements of $e_1\cdot\hat S$ and hence $e_1\cdot\hat S$ is a group.

Suppose now that $s, t \in \hat S$ and that $l(\theta(s\cdot t)) = \infty$. Then, using hypothesis 4, we get that
\begin{eqnarray*}
\theta(s\cdot t) & = & (e_1\cdot s\cdot t, e_1\cdot T(s\cdot t), e_1\cdot T^2(s\cdot t), \dots)\\
          & = & (e_1\cdot s\cdot t, e_1\cdot T(s)\cdot T(t), e_1\cdot T^2(s)\cdot T^2(t), \dots)\\
          & = & (e_1\cdot e_1\cdot s\cdot t, e_1\cdot e_1\cdot T(s)\cdot T(t), e_1\cdot e_1\cdot T^2(s)\cdot T^2(t), \dots)\\
          & = & (e_1\cdot s\cdot e_1\cdot t, e_1\cdot T(s)\cdot e_1 \cdot T(t), e_1\cdot T^2(s)\cdot e_1\cdot T^2(t), \dots).
\end{eqnarray*}
None of the entries of $\theta(s\cdot t)$ are zero, and so $T^i(s), T^i(t) \neq e_0 $ for any $i$ (due to hypothesis 1). Hence $l(\theta(s)) = l(\theta(t)) = \infty$. Furthermore,
\begin{eqnarray*}
\theta(s)\bullet \theta(t) & = & (e_1\cdot s, e_1\cdot T(s), e_1\cdot T^2(s), \dots)\bullet (e_1\cdot t, e_1\cdot T(t), e_1\cdot T^2(t), \dots)\\
          & = & (e_1\cdot s\cdot e_1\cdot t, e_1\cdot T(s)\cdot e_1\cdot T(t), e_1\cdot T^2(s)\cdot e_1\cdot T^2(t), \dots)\\
          & = & \theta(s\cdot t).
\end{eqnarray*}
Similar considerations in the case where $l(\theta(s\cdot t)) < \infty$ show that $\theta$ is a multiplicative map. It is also immediate that $\sigma\circ\theta = \theta \circ T$.

Now we will prove that $\theta$ is injective. A key step to do this is to characterize the $\mathcal{L}$ and $\mathcal{R}$ classes of an element $s\in S$. We do this below.

Let $t\in S$ and suppose that that $l(\theta(t))=n$. If $n=0$, then if follows directly fom the definition of $\theta$ that $t=e_0$ and therefore $t^*=e_0$ and $t\in L(e_0)\cap R(e_0)$. If $1\leq n<\infty$, $n$ is the
smallest integer such that $T^n(t)=e_0$, then it is also the smallest integer such that $T^n(t^*)=e_0$ and, from hypothesis 1, we get that it is also the smallest integer such that $T^n(t\cdot t^*)=T^n(t^*\cdot t)=e_0$. So,
since both $\theta(t\cdot t^*)$ and $\theta(t^*\cdot t)$ are idempotents in $\Sigma_{e_1\cdot \hat S}^\N$ with length $n$, and $e_n$ is the unique idempotent of $S$ whose image under $\theta$ has length $n$, we get that
$t\cdot t^*=t^*\cdot t=e_n$ and hence $t \in L(e_n) \cap R(e_n)$. Finally, if $t\in S$ is such that $l(\theta(t))=\infty$, then $T^i(t)\neq e_0$ and $T^i(t^*)\neq e_0$ for all $i\in\N$, and again from hypothesis 1 we have
that  $T^i(t\cdot t^*)\neq e_0$ and $T^i(t^*\cdot t)\neq e_0$ for all $i\in\N$. Once again, $e_\infty$ is the unique idempotent of $S$ whose image under $\theta$ has length $\infty$ and therefore $t\cdot t^*=t^*\cdot
t=e_\infty$ and $t\in L(e_\infty)\cap R(e_\infty)$.

What we showed above can be summarized by saying that, for all $t\in S$,
\begin{equation}\label{LRclassgroup}
L(t)=R(t)=\{s\in S:\ l(\theta(s))=l(\theta(t))\}
\end{equation}
and $(L(t),\cdot)=(R(t),\cdot)$ is a group.

Getting back to the task of showing that $\theta$ is injective, notice that $S=\bigcup_{e\in E(S)}L(e)$, that is, $S$ a union of disjoint groups. Therefore, for each $e\in E(S)$, the map $\left.\theta\right|_{L(e)}:L(e)\to
\theta(L(e))$ is a group homomorphism and so it is enough to prove that $\left.\theta\right|_{L(e)}$ is injective for each $e\in E(S)$.

Let $e_n \in E(s)$ (where $n$ may be infinite). Notice that the identity in $\theta(L(e_n))$ is the sequence
$$I_n:=\theta(e_n)=(e_1\cdot e_n, e_1\cdot T(e_n), e_1\cdot T^2(e_n),\ldots, e_1\cdot T^{n-1}(e_n))=(\underbrace{e_1,\ldots,e_1}_{n\ times},\o,\o,\dots).$$\\
Now, suppose $n<\infty$ and $s\in L(e_n)$ such that $\theta(s)=I_n$. Then for all $0\leq k\leq n-1$ we have that $e_1\cdot T^k(s)=e_1$ and hence $T^{n-1}(s)=T^{n-1}(e_n\cdot s)=T^{n-1}(e_n)\cdot T^{n-1}(s)=e_1\cdot
T^{n-1}(s)=e_1$. Furthermore, notice that $e_1\cdot T^{n-2}(s)=e_1$ and $T(T^{n-2}(s))=T^{n-1}(s)=e_1$ and so, by hypothesis 5, we have that  $T^{n-2}(s)$ is an idempotent. Since $T(T^{n-2}(s))=T^{n-1}(s)=e_1$ it follows that
$T^{n-2}(s)=e_2$. Proceeding recursively, we obtain that $T^{n-k}(s)=e_k$ for all $1\leq k\leq n$ and hence $s=e_n$ and $\left.\theta\right|_{L(e_n)}$ is injective as desired.

On the other hand, if $s\in L(e_\infty)$ is such that $\theta(s)=I_\infty$, then for all $0\leq k<\infty$ we have that $e_1\cdot T^k(s)=e_1$ which means that $e_{k+1}\cdot s=e_{k+1}$. Hence, for all $k\geq 1$ it follows that $e_k\leq s$, and since $e_\infty$ is the unique upper bound of $\{e_i\}_{i\geq 0}$ it means that $s=e_\infty$.

\end{proof}

Note that part of the proof of Proposition~\ref{chains1} (culminating in \eqref{LRclassgroup}) can be used to get an extension of Theorem 7.5 in \cite[pg. 41]{Clifford1967}, for the case when $E(S)$ is not finite but countable:

\begin{cor}\label{chains2}
Let $(S,\cdot)$ be an inverse monoid with 0 which verifies hypotheses 1-4 of Proposition~\ref{chains1}. Then $S$ is an union of disjoint groups.\end{cor}

\begin{rmk} Note that the image of $\theta$ in Proposition \ref{chains1} is invariant by the shift map, but it is not necessarily a shift space. Furthermore, both Proposition \ref{chains1} and Corollary \ref{chains2} are still
true if either $S$ is an inverse semigroup but not an inverse monoid (in which case the image of $\theta$ does not contain sequences of infinity length) or if $E(S)=\{e_0,\ldots, e_k\}$, for some $k\geq 0$, and $T$ is not surjective
(in which case the image of $\theta$ contains only sequences of length less or equal than $k$).
\end{rmk}

\subsection{General properties of 1-block inverse semigroup shifts induced from group operations}
\label{generalproperties1block}
From now on $(G,\cdot)$ will be a group with infinite cardinality and with identity $1_G$, and $\Lambda\subset\Sigma_G^\NZ$ will be a shift space such that $(\Lambda,\bullet)$ is a inverse subsemigroup of the 1-block inverse semigroup
$(\Sigma_G^\NZ,\bullet)$, where $\bullet$ is induced from $\cdot$.\\

\begin{prop}\label{group-inversesemigroup} Let $(G,\cdot)$ be a group with infinite elements, $(\Lambda,\bullet)\subset(\Sigma_G^\NZ,\bullet)$ be an inverse semigroup shift and let $e^\infty, e^{-\infty}$ be as in Proposition \ref{continuityGroup}. Then:

\begin{enumerate}

\item $(\Lambda,\bullet)$ is a group if one of the following conditions holds:

\begin{enumerate}

\item $\NZ=\N$ and $|\LA|<\infty$;

\item $\NZ=\Z$ and $|\Lambda|<\infty$.

\end{enumerate}

\item $(\Lambda,\bullet)$ is an inverse monoid with zero and without divisors of zero if neither i.(a) nor i.(b) holds.

\end{enumerate}

Furthermore, in case i. above we have that $E(\Lambda)=\{e^\infty\}$, while in case ii. we have two possibilities, namely, $E(\Lambda)=\{e^\infty\ ,\ e^{-\infty}\}$, if $\Lambda$ is row-finite, and $E(\Lambda)=E(\Sigma_G^\NZ)$, if $\Lambda$ is not row-finite.

\end{prop}

\begin{proof}

First, note that in any case, since $\Lambda$ is a shift space, then it contains at least one infinite sequence. Therefore, since it is also an inverse semigroup, given $x=(x_i)_{i\in\NZ}\in\Lambda^{\text{inf}}$ the element $x^*=(x_i^{-1})_{i\in\NZ}$ belongs to $\Lambda$ and then $e^\infty=x\bullet x^*$ also belongs to $\Lambda$.\\

If either i.(a) or i.(b) holds, then $\Lambda$ does not contain any finite sequence and it is a shift space over a finite alphabet. Therefore $E(\Lambda)=\{e^\infty\}$ and, since $(\Lambda,\bullet)$ is an inverse semigroup with an unique idempotent (that is, $e^\infty$), then it is a group shift induced from $(G,\cdot)$ (in particular, under i.(b) we have that $(\Lambda,\bullet)$ is a finite group).\\

Now, if we suppose that neither i.(a) nor i.(b) holds, then $e^{-\infty}=\O\in\Lambda$, that is,  $(\Lambda,\bullet)$ has a zero. Furthermore, since $(\Sigma_G^\NZ,\bullet)$ has no divisors of zero, then neither do $(\Lambda,\bullet)$. Finally, if $\Lambda$ is row-finite, then $\O$ is the unique finite sequence of $\Lambda$, which implies that $E(\Lambda)=\{e^{-\infty},\ e^\infty\}$. On the other hand, if $\Lambda$ is not a row-finite shift, then given a finite sequence $x\in\Lambda^{\text{fin}}\setminus\{\O\}$, with $\l(x)=n\in\NZ$, it follows that $x^*=x=(x_i^{-1})_{i\in\NZ}$ belong to $\Lambda^{\text{fin}}$ and hence $x\bullet x^*=e^n$ belongs to $\Lambda$. Since $\s(\Lambda)=\Lambda$, $e^n\in \Lambda$ implies that $e^k\in\Lambda$ for all $k\in\NZ$ and therefore $E(\Lambda)=E(\Sigma_G^\NZ)$.

\end{proof}

\begin{prop}\label{discontinuousGroup} Let $(\Lambda,\bullet)\subset(\Sigma_G^\NZ,\bullet)$ be a inverse semigroup shift space induced from a group multiplication on $G$. Then $\bullet$ is continuous if, and only if, one of the following conditions holds:
\begin{enumerate}
\item[(a)] $\NZ=\N$ and $|\LA|<\infty$;

\item[(b)] $\NZ=\Z$ and $|\Lambda|<\infty$.
\end{enumerate}
\end{prop}

\begin{proof}{\color{white}.}

As seen in the beginning of the proof of Proposition \ref{group-inversesemigroup}, if $\NZ=\N$ and $|\LA|<\infty$ then $\Lambda$ is a group shift over a finite alphabet. Furthermore, $\Lambda$ is a classical shift (see \cite[Remark 2.24]{Ott_et_Al2014}). Therefore, from \cite{kitchens} it follows that $(\Lambda,\bullet)$ is a topological group shift.

If $\NZ=\Z$ and $|\Lambda|<\infty$ then by Proposition \ref{group-inversesemigroup} $(\Lambda,\bullet)$ is a finite group. Furthermore, the topology induced on $\Lambda$ from $\Sigma_G^\NZ$ corresponds to the discrete topology and, therefore, $\bullet$ is a continuous operation.

For the converse, suppose that neither (a) or  (b) holds. Then, by the proof of Proposition \ref{group-inversesemigroup}, $(\Lambda,\bullet)$ is an inverse monoid with zero. Let $\Phi:\Lambda\boxtimes\Lambda\to\Lambda$ be the map defined by $\Phi(x,y):=x\bullet y$, where $(x,y)=\left( (x_i,y_i)\right) \in \Lambda\boxtimes\Lambda$ and $\Phi(\O)=\O$. Notice that, by Proposition \ref{map_wp}, we just need to prove that $\Phi$ is not continuous.

In the case $\NZ=\N$ and $|\LA|=\infty$, since $\{ (b,c)\in \LA \boxtimes \LA: b c =1_G \}=\{
(b,b^{-1})\in \LA \boxtimes \LA: b\in \LA \}\cong \LA$, it follows from Theorem \ref{continuityfinite} that $\Phi$ is not continuous.

Now suppose that $\NZ=\Z$ and $|\Lambda|=\infty$. If there exists a non-constant sequence $x\in\Lambda^{\text{inf}}$, then it follows that $z^0:=(x,x^*)\in(\Lambda\boxtimes\Lambda)^{\text{inf}}$ is such that $z^n:=\s^n(z^0)\to\O$ as $n\to\infty$.
However $\Phi(z^n)=\s^n(x)\bullet\s^n(x^*)=\s^n(x\bullet x^*)=\s^n(e^\infty)=e^\infty$ for all $n\geq 0$. If there are no non-constant sequences in $\Lambda$ then any sequence in $\Lambda$ is constant. Since $|\Lambda|=\infty$, this means that we can take a sequence $(z^n)_{n\geq 1}\in(\Lambda\boxtimes\Lambda)^{\text{inf}}$ such that $z^n:=(x^n,x^{n*})$ and $z^m\neq z^n$ for all $m\neq n$. It follows that $z^n\to \O$ as $n\to\infty$, but again $\Phi(z^n)=x^n\bullet x^{n*}=e^\infty$ for all $n\geq 1$.\\

\end{proof}

For what comes next we need the following definitions.

\begin{defn}\label{sets_of_words_1} Given $\Lambda\subset\Sigma_A^\NZ$ and $n\in\N$, we define $$\mathfrak{B}_n(\Lambda):=B_n(\Lambda)\cap A^n$$ and
$$\mathfrak{B}(\Lambda):=\bigcup_{n\geq 1}\mathfrak{B}_n(\Lambda).$$
\end{defn}

\begin{defn}\label{sets_of_words_2} Given $\Lambda\subset\Sigma_A^\NZ$,  $a\in \mathfrak{B}(\Sigma_A^\NZ)$, and $k\in\N$, we define $$\mathfrak{F}_k(\Lambda,a):=\F_k(\Lambda,a)\cap A^k=\{b\in \mathfrak{B}_k(\Lambda):\ ab\in \mathfrak{B}(\Lambda)\}$$ and
$$\mathfrak{P}_k(\Lambda,a):=\Pp_k(\Lambda,a)\cap A^k=\{b\in \mathfrak{B}_k(\Lambda):\ ba\in\mathfrak{ B}(\Lambda)\}.$$
\end{defn}
We note that words in $B_n(\Lambda), \F_n(\Lambda,a),$ and $\Pp_n(\Lambda,a)$ may contain the empty letter, while those in $\mathfrak{B}_n(\Lambda), \mathfrak{F}_n(\Lambda,a),$ and $\mathfrak{P}_n(\Lambda,a)$ do not.

\begin{defn}\label{sets_of_words_3}
For each $n\in\NZ\cup\{\pm\infty\}$ define $$\Lambda_n:=\{x\in\Lambda:\ l(x)=n\}.$$ We consider on $\mathfrak{B}_n(\Lambda)$ the product topology of $A$, while on $\Lambda_n$ we consider the topology induced from $\Lambda$ (which is homeomorphic to the product topology of $A$ if $\NZ=\N$).
\end{defn}

Due to Propositions \ref{group-inversesemigroup} and \ref{discontinuousGroup}, (for the more interesting cases) when $(\Lambda,\bullet)$ is a group shift, the operation $\bullet$ is not continuous on $\Lambda$ . However, if we consider the restriction of $\bullet$ on $\Lambda_n$ or the piecewise operation on $\mathfrak{B}_n(\Lambda)$ (which we will also denote as $\cdot$), then the following results can be easily proven:

\begin{prop} Let $(\Lambda,\bullet)\subset(\Sigma_G^\NZ,\bullet)$ be an inverse semigroup shift space induced from a group multiplication on $G$. For each $n\in\N$, we have that $(\mathfrak{B}_n(\Lambda),\cdot)$ is a topological group with identity equal to $1_G^n$, the block with all its $n$ entries equal to $1_G$. Furthermore, if $a\in \mathfrak{B}_n(\Lambda)$
then $a^{-1}$ is the block obtained by taking the inverses of each entry of $a$.
\end{prop}

\begin{prop} Let $(\Lambda,\bullet)\subset(\Sigma_G^\NZ,\bullet)$ be an inverse semigroup shift space induced from a group multiplication on $G$. For each $n\in\NZ\cup\{\pm\infty\}$, we have that $(\Lambda_n,\bullet)$ is a topological group where the identity is $e^n$. Furthermore, if $x\in \Lambda_n$
then $x^{-1}$ is the sequence obtained by taking the inverses of each entry (distinct of $\o$) of $x$.
\end{prop}

Note that  $(\mathfrak{B}_n(\Lambda),\cdot)$ is a subgroup of $(G^n,\cdot)$, that agrees with
$(G^n,\cdot)$ only when $\Lambda$ is the full shift.\\

\begin{prop}\label{NormalSubgroup} For each $k,n\geq 1$, we have that $\FF_k(\Lambda,1_G^n)$ and $\Ppp_k(\Lambda,1_G^n)$ are normal subgroups of $(\mathfrak{B}_k(\Lambda),\cdot)$.
\end{prop}

\begin{proof}
We will prove the result only for the follower sets, as the proof for the predecessor sets is analogous.

First we show that $\FF_k(\Lambda,1_G^n)$ is a subgroup of $\mathfrak{B}_k(\Lambda)$. For this, notice that given $g=(g_1,\ldots,g_k),h=(h_1,\ldots,h_k)\in\FF_k(\Lambda,1_G^n)$ there exist two sequences  $x,y\in\Lambda$ such that
$$\begin{array}{ll}
x_i=y_i=1_G,& \forall \ 1\leq i\leq n,\\
x_i=g_{i-n},& \forall \ n+1\leq i\leq n+k,\\
y_i=h_{i-n},& \forall \ n+1\leq i\leq n+k.
\end{array}$$
So, $x\bullet y^{-1}$ is a sequence in $\Lambda$ such that
$$
(x\bullet y^{-1})_i=x_i\cdot y_i^{-1}=\left\{\begin{array}{ll}1_G,& \forall \ 1\leq i\leq n,\\
g_{i-n}\cdot h_{i-n}^{-1},& \forall \ n+1\leq i\leq n+k
\end{array}\right.$$
and hence $g\cdot h^{-1}=(g_1\cdot h_1^{-1},\ldots,g_k\cdot h_k^{-1})\in\FF_k(\Lambda,1_G^n)$.

Now we show that $\FF_k(\Lambda,1_G^n)$ is a normal subgroup of $\mathfrak{B}_k(\Lambda)$. So, let $a=(a_1\ldots,a_k)\in \mathfrak{B}_k(\Lambda)$ and $g=(g_1,\ldots,g_k)\in\FF_k(\Lambda,1_G^n)$ and take sequences $x,y\in \Lambda$ such that
$$\begin{array}{ll}
x_i=1_G, & \forall \ 1\leq i\leq n,\\
x_i=g_{i-n},& \forall \ n+1\leq i\leq n+k,\\
y_i=a_{i-n},& \forall \ n+1\leq i\leq n+k.
\end{array}$$
Then $y\bullet x\bullet y^{-1}\in\Lambda$ and is such that
$$\begin{array}{lcl}
(y\bullet x\bullet y^{-1})_i=y_i\cdot x_i\cdot y_i^{-1}=\left\{\begin{array}{ll} y_i\cdot 1_G\cdot y_i^{-1}=1_G,& \forall \ 1\leq i\leq n,\\
a_i\cdot g_i\cdot a_i^{-1},& \forall \ n+1\leq i\leq n+k\end{array}\right.
\end{array}$$
and therefore $a\cdot g\cdot a^{-1}=(a_1\cdot g_1\cdot a_1^{-1},\ldots, a_k\cdot g_k\cdot a_k^{-1})\in\FF_k(\Lambda,1_G^n)$ as desired.

\end{proof}

We now prove a key result relating the follower and predecessor sets of arbitrary words to the follower and predecessor sets of words consisting only of the identity of $G$.
\begin{prop}\label{FollowerSetProduct}
For each $k,n\geq 1$ and $a\in \mathfrak{B}_n(\Lambda)$ we have that
\begin{enumerate}
\item $b\in\FF_k(\Lambda,a)$ if, and only if, $b\cdot \FF_k(\Lambda,1_G^n)=\FF_k(\Lambda,1_G^n)\cdot b=\FF_k(\Lambda,a)$.

\item $b\in\Ppp_k(\Lambda,a)$ if, and only if, $b\cdot \Ppp_k(\Lambda,1_G^n)=\Ppp_k(\Lambda,1_G^n)\cdot b=\Ppp_k(\Lambda,a)$.
\end{enumerate}

\end{prop}

\begin{proof} We will prove the result only for the follower sets and for the multiplication on the left, since the other cases can be proved with analogous arguments.

Let $b=(b_1,\ldots,b_k)\in\FF_k(\Lambda,a)$ and $c=(c_1,\ldots,c_k)\in \FF_k(\Lambda,1_G^n)$. Then there exist $x,y\in\Lambda$ such that
$$\begin{array}{ll}
x_i=a_i, & \forall \ 1\leq i\leq n,\\
x_i=b_{i-n},& \forall \ n+1\leq i\leq n+k,\\
y_i=1_G,& \forall \ 1\leq i\leq n,\\
y_i=c_{i-n},& \forall \ n+1\leq i\leq n+k.\\
\end{array}$$
So, $x\bullet y\in\Lambda$ is such that
$$\begin{array}{ll}
(x\bullet y)_i=a_i\cdot 1_G= a_i,& \forall \ 1\leq i\leq n,\\
(x\bullet y)_i=b_{i-n}\cdot c_{i-n},& \forall \ n+1\leq i\leq n+k,\\
\end{array}$$
and hence $b\cdot c\in\FF_k(\Lambda,a)$ and $b\cdot \FF_k(\Lambda,1_G^n)\subseteq\FF_k(\Lambda,a)$.

For the other inclusion, let again $b=(b_1,\ldots,b_k)\in\FF_k(\Lambda,a)$ and let $x\in \Lambda$ be as above. Then $x^{-1}$ is such that
$$\begin{array}{ll}
x^{-1}_i=a^{-1}_i,& \forall \ 1\leq i\leq n,\\
x^{-1}_i=b^{-1}_{i-n},& \forall \ n+1\leq i\leq n+k\\
\end{array}$$
and hence $b^{-1}\in\FF_k(\Lambda,a^{-1})$, that is, $b\in\FF_k(\Lambda,a)$ if and only if $b^{-1}\in\FF_k(\Lambda,a^{-1})$.

Now, given $d\in\FF_k(\Lambda,a)$, let $z\in\Lambda$ be such that
$$\begin{array}{ll}
z_i=a_i,& \forall \ 1\leq i\leq n,\\
z_i=d^{-1}_{i-n},& \forall \ n+1\leq i\leq n+k.\\
\end{array}$$
Then $x^{-1}\bullet z$ is such that
$$\begin{array}{ll}
(x^{-1}\bullet z)_i=a^{-1}_i\cdot a_i= 1_G,& \forall \ 1\leq i\leq n,\\
(x^{-1}\bullet z)_i=b^{-1}_{i-n}\cdot d_{i-n},& \forall \ n+1\leq i\leq n+k,\\
\end{array}$$
which implies that $b^{-1}\cdot d\in \FF_k(\Lambda,1_G^n)$, that is, $b^{-1}\cdot \FF_k(\Lambda,a)\subseteq\FF_k(\Lambda,1_G^n)$.

Finally, notice that
$$\FF_k(\Lambda,a)=b\cdot b^{-1}\cdot \FF_k(\Lambda,a)\subseteq b\cdot \FF_k(\Lambda,1_G^n)\subseteq \FF_k(\Lambda,a)$$
and hence $b\cdot \FF_k(\Lambda,1_G^n)=\FF_k(\Lambda,a)$ as desired.

For the converse, since $1_G^k\in \FF_k(\Lambda,1_G^n)$, we promptly obtain that $b=b\cdot 1_G^k\in b\cdot \FF_k(\Lambda,1_G^n)=\FF_k(\Lambda,a)$.

\end{proof}

\begin{cor}\label{LAF_and_LAP}
For all $k,n\geq 1$ we have that the families $\LAF^{n,k}:=\{\FF_k(\Lambda,a):\ a\in \mathfrak{B}_n(\Lambda)\}$ and $\LAP^{n,k}:=\{\Ppp_k(\Lambda,a):\ a\in \mathfrak{B}_n(\Lambda)\}$ are pairwise disjoint. Furthermore, $(\LAF^{n,k},\cdot)$ and
$(\LAP^{n,k},\cdot)$ are groups (where $\cdot$ is the piecewise operation induced from the operation $\cdot$ on the alphabet $\LA$) and, for all $a,b\in \mathfrak{B}_n(\Lambda)$, we have that

\begin{enumerate}
\item $|\FF_k(\Lambda,a)|= |\FF_k(\Lambda,b)|$ and $|\Ppp_k(\Lambda,a)|= |\Ppp_k(\Lambda,b)|$;

\item $\FF_k(\Lambda,a)\cdot \FF_k(\Lambda,b)=\FF_k(\Lambda,a\cdot b)$ and $\Ppp_k(\Lambda,a)\cdot \Ppp_k(\Lambda,b)=\Ppp_k(\Lambda,a\cdot b)$.
\end{enumerate}

\end{cor}

\begin{proof}
The above statements follow directly from Proposition \ref{NormalSubgroup} and Theorem \ref{FollowerSetProduct}, which imply that $\LAF^{n,k}$ and $\LAP^{n,k}$ are the families of cosets of $\FF_k(\Lambda,1_G^n)$ and $\Ppp_k(\Lambda,1_G^n)$, respectively.

\end{proof}

\begin{prop}
For all $k,n\geq 1$ we have that $(\LAF^{n,k},\cdot)$ and $(\LAP^{k,n},\cdot)$ are group isomorphic.
\end{prop}

\begin{proof}
First notice that given $n,k\geq 1$ and $b,b'\in \mathfrak{B}_k(\Lambda)$, we have that there exists $a\in \mathfrak{B}_n(\Lambda)$ such that $b,b'\in\FF_k(\Lambda,a)$ if and only if $a\in\Ppp_n(\Lambda,b)\cap\Ppp_n(\Lambda,b')$, which happens if and only if
$\Ppp_n(\Lambda,b)=\Ppp_n(\Lambda,b')$,
where the second equivalence is due to the fact that the elements of $\LAP^{k,n}$ are pairwise disjoint. With this in mind, we now obtain a well-defined onto map
$$\begin{array}{lcl}\tau:&\LAF^{n,k}&\to\LAP^{k,n}\\ &\FF&\mapsto \Ppp_n(\Lambda,b)\end{array},$$
where $b\in\FF$ is an arbitrary element.

To check that $\tau$ is one-to-one, suppose that $\tau(\FF_k(\Lambda,a))=\tau(\FF_k(\Lambda,a'))$ for some $a,a'\in \mathfrak{B}_n(\Lambda)$. Then there exists $b\in \mathfrak{B}_k(\Lambda)$ such that
$\tau(\FF_k(\Lambda,a))=\tau(\FF_k(\Lambda,a'))=\Ppp_n(\Lambda,b)$ and, from the definition of $\tau$, this means that $b\in \FF_k(\Lambda,a)$ and $b\in \FF_k(\Lambda,a')$ and hence $\FF_k(\Lambda,a)=\FF_k(\Lambda,a')$.

Finally, let $a,a'\in \mathfrak{B}_n(\Lambda)$, $b\in\FF_k(\Lambda,a)$ and $b'\in\FF_k(\Lambda,a')$. Then we have that $b\cdot b'\in \FF_k(\Lambda,a\cdot a')$ and hence
$$\tau(\FF_k(\Lambda,a)\cdot\FF_k(\Lambda,a'))=\tau(\FF_k(\Lambda,a\cdot a'))=\Ppp_n(\Lambda,b\cdot b')=\Ppp_n(\Lambda,b)\cdot \Ppp_n(\Lambda,b')=
\tau(\FF_k(\Lambda,a))\cdot\tau(\FF_k(\Lambda,a')).$$

\end{proof}

We immediately obtain the following corollary.

\begin{cor}\label{LAF_isomorphic_LAP}
For all $k,n\geq 1$ we have that $|\LAF^{n,k}|=|\LAP^{k,n}|$.
\end{cor}

Notice that in the finite alphabet case an inverse semigroup shift $\Lambda$ is always a shift of finite type (this can be proved using the argument of Proposition 4 in \cite{kitchens}, in spite of the fact that a two-sided shift $\Lambda$ may contain the empty sequence $\O$), while this is not necessarily
true in the infinite alphabet case (see examples below). Furthermore, in the finite alphabet case, Corollaries \ref{LAF_and_LAP} and \ref{LAF_isomorphic_LAP} imply that, given $k,n\geq 1$, if $a\in \mathfrak{B}_n(\Lambda)$ and $b\in
\mathfrak{B}_k(\Lambda)$ then $\FF_k(\Lambda,a)$ and $\Ppp_n(\Lambda,b)$ have the same cardinality, while the same is not necessarily true for the infinite alphabet case (see Example \ref{ex_fractal} below).

Next we illustrate Theorem \ref{FollowerSetProduct} and the discussion above with a few examples of 1-block inverse semigroup shifts.

\begin{ex}\label{Example_F_infinite1}
Let $G=\Z$ with the usual sum of integers. Let $\Lambda' \subset\Sigma_G^\NZ$ be the set of infinite sequences of integers $(x_i)_{i\in\NZ}$ such that $x_i$ and $x_{i+2}$ have the same parity for all $i\in\NZ$. Then $\Lambda:=cl(\Lambda')$,
the closure of $\Lambda'$, is a 2-step shift, but not a shift of finite type. Furthermore, let $\bullet$ be the piecewise operation induced from the group operation in $\Z$. Then $(\Lambda,\bullet)$ is an inverse semigroup.

Notice that, for all $a\in L_{\Lambda}$, $\FF_1(\Lambda,a)=\Z$. Also, for $k\geq 2$, $\FF_k(\Lambda,2i)=\FF_k(\Lambda,2j)\subseteq \Z^k$ and $\FF_k(\Lambda,2i+1)=\FF_k(\Lambda,2j+1)\subseteq \Z^k$ for all $i,j\in \Z$.
Furthermore, denoting by $E$ and $O$ the sets of even and odd integers, respectively, it follows that, for each $k\geq 1$ and $n\geq 2$, given $a=(a_1,\ldots,a_n)\in \mathfrak{B}_n(\Lambda)$ we have that

$$\FF_k(\Lambda,a)=\left\{\begin{array}{ccl} E^k, &\text{ if}& a_{n-1}\text{ and } a_n\text{ are even,}\\
                                     O^k, &\text{ if}& a_{n-1}\text{ and } a_n\text{ are odd,}\\
                                     \underbrace{E\times O\times\ldots,}_{k\ alternated\ Cartesian\ products} &\text{ if}& a_{n-1}\text{ is even and } a_n\text{ is odd,}\\
                                     \underbrace{O\times E\times\ldots,}_{k\ alternated\ Cartesian\ products} &\text{ if}& a_{n-1}\text{ is odd and } a_n\text{ is even.}\\
                                     \end{array}\right. $$
So any follower set has infinitely many elements, while $\LAF^{1,1}$ contains exactly one element, $\LAF^{1,k}$ contains two elements for $k\geq 2$ and, for $n\geq 2$, $\LAF^{n,k}$ contains exactly four elements.
\end{ex}

\begin{ex}\label{Example_F_infinite2} Let $G=\Z$ with the usual sum of integers. Let $\Lambda' \subset\Sigma_G^\Z$ be the set of all periodic infinite sequences, that is, $\Lambda'=\{x\in\Sigma_A^\Z:\ \exists n\in\N,\ s^n(x)=x\}$. Thus $\Lambda:=cl(\Lambda')=\Lambda'\cup\{\O\}$, with the piecewise operation $\bullet$ induced from the group operation in $\Z$, is an inverse semigroup. Note that $\Lambda$ is an {\em infinite-step shift} where the set of forbidden words is exactly all non periodic sequences of $B_{\text{linf}}$. In particular, $\FF_k(\Lambda,a)=\Z$ for all $a=(a_1,\ldots,a_n)\in \mathfrak{B}_n(\Lambda)$ and $k\geq 1$.
\end{ex}

\begin{ex}\label{ex_ZxZ}
Let $G=\Z^2$ with the usual sum and define $\Lambda\subset\Sigma_G^\NZ$ as the inverse semigroup shift with the induced operation, where the infinite sequences $(x_i)_{i\in\NZ}=(x_{1,i},x_{2,i})_{i\in\NZ}$ are such that for
any $i\in\NZ$, $x_{2,i+1}=x_{2,i}$, that is,
$$\Lambda^{\text{inf}}=\bigcup_{i\in\Z}(\Z\times\{i\})^{\NZ}$$
Here, $\Lambda$ is a 1-step shift but not a shift of finite type. For any $k,n\geq 1$, given $a=(a_{1,j},a_{2,j})_{1\leq j\leq n}\in \mathfrak{B}_n(\Lambda)$ we have that $\FF_k(\Lambda,a)=(\Z\times\{a_{2,n}\})^k$. Hence any follower set has infinitely many elements and, for $k,n\geq 1$, the space
$\LAF^{n,k}$ also has infinitely many elements.

\end{ex}

Our next example makes use of direct limits of groups so, for the reader's convenience, we recall the definition below.

For each $i\in\N$ let $G_i$ be a group and $\gamma_i:G_i\to G_{i+1}$ a homomorphism. If $i<j$, the map $\gamma_{j-1}\circ\gamma_{j-2}\circ\cdots\circ\gamma_{i+1}\circ\gamma_i$ is a homomorphism from $G_i$ to $G_j$; we
shorten this to
$$
\gamma_{i, j} := \gamma_{j-1}\circ\gamma_{j-2}\circ\cdots\circ\gamma_{i+1}\circ\gamma_i.$$

Recall that the direct limit $G$ associated to $\{G_i, \gamma_i\}$ is the set of equivalence classes $[g,i]$, where $i\in \N$ and $g\in G_i$, and $[g, i] = [h, j]$ if, and only if, there exists $k \geq i,j$ such that
$\gamma_{i, k}(g) = \gamma_{j,k}(h)$. This is a group under the following operation:
$$[g,i]\cdot[h,j] = [\gamma_{i,j}(g)+ h, j], \text{ if } i\leq j\
$$
and $$ [g,i]\cdot[h,j] = [g+\gamma_{j,i}(h), i]\text{ if } i\geq j,
$$
with inverse given by
$[g,i]^{-1} = [g^{-1},i]$ and identity $\mathfrak{e}=[0,1]$.

One really thinks of $G$ as the disjoint union of the groups, with elements identified if they are eventually mapped to the same elements by $\gamma$. We say that an element $\mathfrak{g}\in G$ is represented in $G_i$
if $\mathfrak{g} = [g,i]$ for some element $g\in G_i$. Clearly, if $\mathfrak{g}$ is represented in $G_i$, it is represented in $G_j$ for all $j>i$. Given $\mathfrak{g}\in G$, we say that $[g,i]$ is the first representation of $\mathfrak{g}$ if  $i\geq 1$ is the least index such that
$\mathfrak{g}=[g,i]$, that is, $i$ the first time that $\mathfrak{g}$ is represented in the family $\{G_k\}_{k\geq 1}$.

\begin{ex}\label{Example_Lim-Direct} For $i\in\N$, let $G_i$ be the finite group $\Z_{2^i}$ with addition and let $\gamma_i:G_i\to G_{i+1}$ be the homomorphism given by $\gamma_i(g):=2g$ (where the product is the usual
product of integers) and consider the resulting direct limit group $G$, which is an infinite group. Note that the first representation of the identity is $\mathfrak{e}=[0,1]$, while all other elements of $G$ correspond to
first representations of the form $[g,i]$, with $g$ odd and $i\geq 1$.

Let $H\lhd G$ be a finite normal subgroup of $G$ (for example $H:=\{\mathfrak{g}\in G:\ \mathfrak{g} = [g,i]\}$ for some fixed $i$). For each letter $\mathfrak{g}\in G$, define
$$\rho(\mathfrak{g}):=\mathfrak{g}\cdot H.$$
Note that $\rho(\mathfrak{e})=H$ and, for each $\mathfrak{g}\in G$, $\rho(\mathfrak{g})$ is a lateral class of $H$ (and then  $\LAF^{1,1}=\{\rho(\mathfrak{g}):\ \mathfrak{g}\in G\}$ is a partition of $G$ by finite sets).

Now, let $\Lambda\subset\Sigma_G^\NZ$ be the shift space where $$\Lambda^{\text{inf}}:=\{(\mathfrak{g}_i)_{i\in\NZ}:\ \mathfrak{g}_{i+1}\in \rho(\mathfrak{g}_i)\}.$$

Clearly $\Lambda$ is a 1-step row-finite shift such that $\FF_1(\Lambda,\mathfrak{g})=\rho(\mathfrak{g})$ (and thus it is not a shift of finite type). Furthermore, to show that $\Lambda$ is an inverse semigroup with the 1-block induced operation, it is sufficient to show that the set $\mathfrak{B}_2(\Lambda)$ is closed under the 1-block operation. So, let $(\mathfrak{g}_1,\mathfrak{g}_2),(\mathfrak{h}_1,\mathfrak{h}_2)\in \mathfrak{B}_2(\Lambda)$. Then
$$(\mathfrak{g}_1,\mathfrak{g}_2)\cdot(\mathfrak{h}_1,\mathfrak{h}_2)=(\mathfrak{g}_1\cdot\mathfrak{h}_1,\mathfrak{g}_2\cdot\mathfrak{h}_2)$$
and, since $\mathfrak{g}_2\in \rho(\mathfrak{g}_1)=\mathfrak{g}_1\cdot H$ and $\mathfrak{h}_2\in \rho(\mathfrak{h}_1)=\mathfrak{h}_1\cdot H$, we have that $$\mathfrak{g}_2\cdot\mathfrak{h}_2\in \mathfrak{g}_1\cdot H\cdot
\mathfrak{h}_1\cdot H=\mathfrak{g}_1\cdot\mathfrak{h}_1\cdot H=\rho(\mathfrak{g}_1\cdot\mathfrak{h}_1)$$ as desired.

\end{ex}

\begin{ex}\label{ex_fractal} Let $G$ be the group given in Example \ref{Example_Lim-Direct} above and let $H:=\{\mathfrak{g}\in G:\ \mathfrak{g} = [g,1]\}=\{[0,1],\ [1,1]\}$. Given $\mathfrak{g} = [g,i]\in G$ define its
follower set as $$\FF_1(\Lambda,\mathfrak{g}):=[g,i+1]\cdot H=\{[g,i+1],\ [2^i+g,i+1]\},\qquad\text{where }2^i+g\text{ stands for the sum }mod\ 2^{i+1}.$$
Now, define $\Lambda\subset\Sigma_G^\NZ$ as the shift such that
$$\Lambda^{\text{inf}}:=\{(\mathfrak{g}_i)_{i\in\NZ}:\ \mathfrak{g}_{i+1}\in \FF_1(\Lambda,\mathfrak{g}_i)\},$$
which is a row-finite Markovian edge shift whose infinite sequences are obtained from bi-infinite walks on the edge graph presented in Figure \ref{ex_graph1} (we leave the proof that $\Lambda^{\text{inf}}$ is a group to the
reader).

In this case, $\LAF^{1,1}$ has infinitely many follower sets, each of them containing only two elements. Furthermore, this provides an example of an inverse semigroup shift where $\Ppp_1(\Lambda,[0,1])$ and $\FF_1(\Lambda,[0,1])$
have distinct cardinalities, which can not occur for group shifts over finite alphabets (see discussion after Corollary \ref{LAF_isomorphic_LAP}).

\begin{figure}[h]
\centering
\includegraphics[width=0.4\linewidth=1.0]{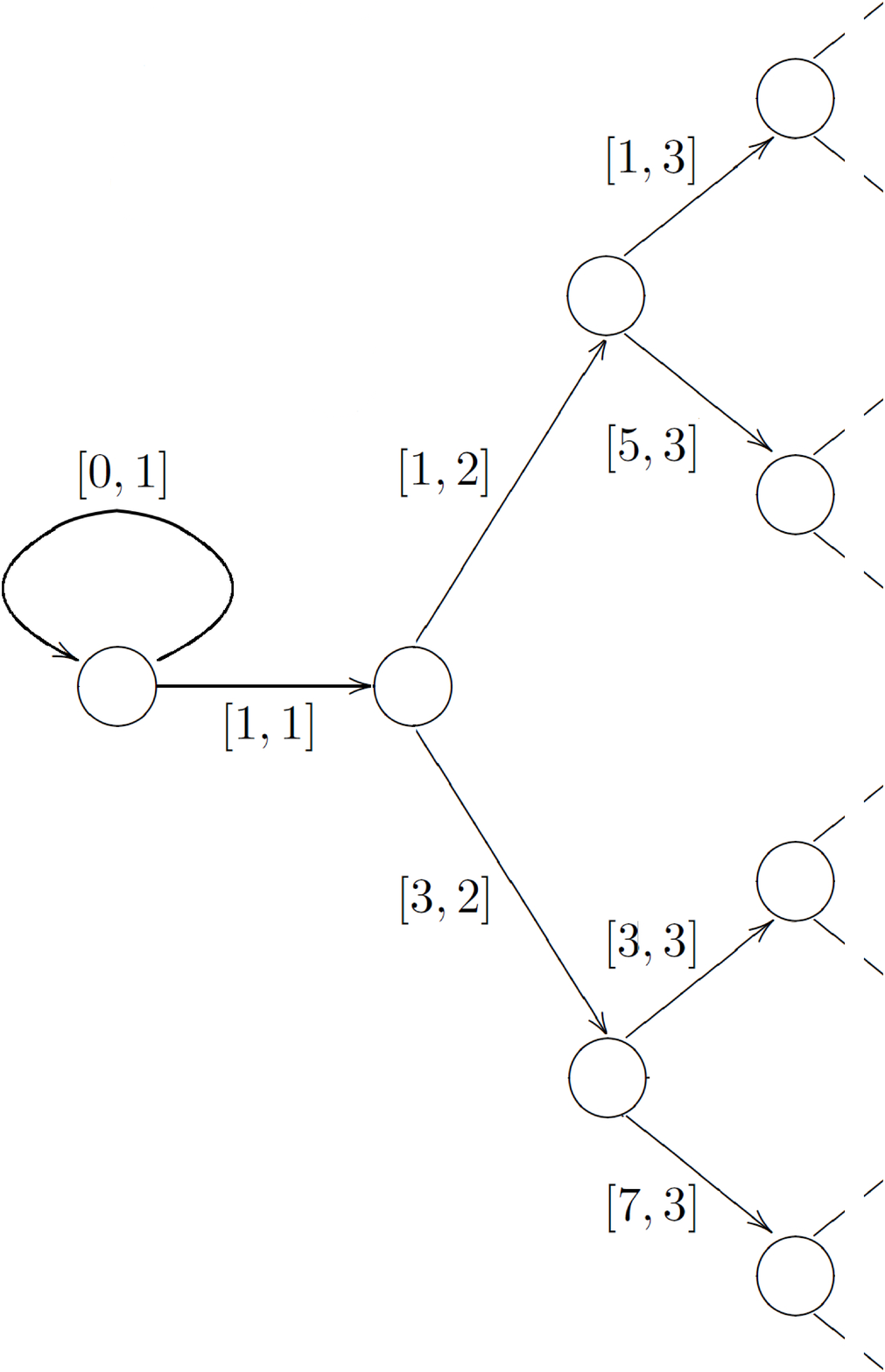}
\caption{The edge graph which represents the row-finite Markov shift given in Example \ref{ex_fractal}.
Each edge corresponds to the first representation of one element of $G$.}\label{ex_graph1}
\end{figure}

\end{ex}

\section{Isomorphism of two-sided Markovian 1-block inverse semigroup shifts induced from group operations}\label{Isomorphism of two-sided Markovian}
\label{final section}
In this section, as in the previous one, we assume that $\Lambda\subset\Sigma_G^\NZ$ is a shift space with $\LA=G$ and also assume that $(\Lambda,\bullet)$ is a subsemigroup of the 1-block inverse semigroup $(\Sigma_G^\NZ,\bullet)$,
where $\bullet$ is induced from a group $(G,\cdot)$.

The next theorem gives sufficient conditions for a 1-block inverse semigroup shift $\Lambda$ to be an $M$-step shift and to be topologically conjugate and isomorphic (as semigroup) to a row-finite shift where an 1-block inverse semigroup operation is defined.

\begin{theo}\label{Conditions_M-step} Let $\Lambda\subset\Sigma_G^\Z$ be a 1-block inverse semigroup shift induced from the group $(G,\cdot)$. If there exists $M\geq 1$ such that
\begin{enumerate}
\item[(a)]  $\FF_1(\Lambda,1_G^m)=\FF_1(\Lambda,1_G^M)$ for all $m\geq M$, then $\Lambda$ is an $M$-step shift and $\Lambda$ is topologically conjugate and isomorphic to an edge shift with 1-block operation (via the $M^{th}$-higher block code);

\item[(b)] $\FF_1(\Lambda,1_G^M)$ is finite, then $\Lambda$ is $N$-step, for some $N\geq M$, and $\Lambda$ is topologically conjugate and isomorphic to a row-finite edge shift with 1-block operation (via the $N^{th}$-higher block code).
\end{enumerate}
\end{theo}

\begin{proof}{\color{white}.}

\begin{enumerate}
\item[(a)]  Let $M\geq 1$ be such that $\FF_1(\Lambda,1_G^m)=\FF_1(\Lambda,1_G^M)$ for all $m\geq M$.  Let $\mathbf{F}\subset B(\Sigma_G^\Z)$ be a set of forbidden words such that $\Lambda=X_\mathbf{F}$. Note that we can
    always choose $\mathbf{F}$ such that, for all $w\in\mathbf{F}$, $l(w)\geq M+1$ and $w$ doesn't contain any proper subblock which belongs to $\mathbf{F}$.

Let us show that all blocks of $\mathbf{F}$ have length $M+1$. Assume, by contradiction, that there exists $w\in\mathbf{F}$ such that $l(w)=m>M+1$. Since $w$ does not contain subblocks belonging to $\mathbf{F}$, we
can write $w=uv$, the concatenation of two blocks $u=(u_1,\ldots,u_k)\in \mathfrak{B}_k(\Lambda)$, with $k=m-M-1$, and $v=(v_1,\ldots,v_{M+1})\in \mathfrak{B}_{M+1}(\Lambda)$. Now, since
$\FF_1(\Lambda,(u_1,\ldots,u_k,v_1,\ldots,v_M))\subset\FF_1(\Lambda,(v_1,\ldots,v_M))$, we can find  $$b\in \FF_1(\Lambda,(u_1,\ldots,u_k,v_1,\ldots,v_M))\cap\FF_1(\Lambda,(v_1,\ldots,v_M)).$$
   From Theorem \ref{FollowerSetProduct}, we have that $b\cdot \FF_1(\Lambda,1_G^M)=\FF_1(\Lambda,(v_1,\ldots,v_M))$ and $b\cdot \FF_1(\Lambda,1_G^{m-1})=\FF_1(\Lambda,(u_1,\ldots,u_k,v_1,\ldots,v_M))$. So,
   $$v_{M+1}\in  \FF_1(\Lambda,(v_1,\ldots,v_M))=b\cdot \FF_1(\Lambda,1_G^M)=b\cdot \FF_1(\Lambda,1_G^{m-1})=\FF_1(\Lambda,(u_1,\ldots,u_k,v_1,\ldots,v_M)),$$
   which means that $w=(u_1,\ldots,u_k,v_1,\ldots,v_M,v_{M+1})$ is an allowed block, contradicting the assumption that $w\in\mathbf{F}$.\\

   The second part of the statement follows directly from \cite[Proposition 3.19]{GoncalvesSobottkaStarling2015_2} by applying the $M^{th}$-higher block code on $\Lambda$
    and using it to induce an operation on $\Lambda^{[M]}$.

\item[(b)] Let $M\geq 1$ be such that $\FF_1(\Lambda,1_G^M)$ is finite. Then, since for any $m\geq M$ we have that $\FF_1(\Lambda,1_G^{m+1})\subseteq \FF_1(\Lambda,1_G^{m})\subseteq \FF_1(\Lambda,1_G^M))$, there exists $N\geq
    M$ such that $\FF_1(\Lambda,1_G^n)=\FF_1(\Lambda,1_G^N))$ for all $n\geq N$. So, from part (a) above, we have that $\Lambda$ is an $N$-step shift and it is topologically conjugate and isomorphic to an edge shift with 1-block operation (via the
    $N^{th}$-higher block code). Furthermore, from Corollary \ref{LAF_and_LAP}, we have that
   $|\FF_1(\Lambda,a))|=|\FF_1(\Lambda,1_G^N))|<\infty$ for all $a\in \mathfrak{B}_N(\Lambda)$, which implies that the $N^{th}$ higher block shift of $\Lambda$ is a row-finite shift.

\end{enumerate}
\end{proof}

\begin{ex} Let $G$ be the group given in Example \ref{Example_Lim-Direct} and, for each $\mathfrak{g}\in G$, let $\rho(\mathfrak{g})$ be defined as in that example. Define $\Lambda\subset\Sigma_G^\Z$ as the shift such that
$$\Lambda^{\text{inf}}:=\{(\mathfrak{g}_i)_{i\in\N}:\ \mathfrak{g}_{i+2}\in \rho(\mathfrak{g}_i)\}.$$
Then $\Lambda$ is an inverse semigroup shift (with the corresponding induced operation). Note that $\Lambda$ is not a row-finite shift since, for each $\mathfrak{g}\in G$, we have that $\FF_1(\Lambda,\mathfrak{g})=G$, an
infinite set. However, for each $(\mathfrak{g,h})\in \mathfrak{B}_2(\Lambda)$ we have that $\FF_1(\Lambda,(\mathfrak{g,h}))=\rho(\mathfrak{g})$, a finite set. Moreover, for any $m\geq 2$, the word $1_G^m$ is such that
$\FF_1(\Lambda,1_G^m)=\rho(1_G)$ and thus we can take $N=2$ in Theorem \ref{Conditions_M-step}.(b) and hence the $2^{nd}$ higher block presentation of $\Lambda$ is a 1-block inverse semigroup shift which is a 1-step row-finite
shift.
\end{ex}

We now characterize Markovian 1-block inverse semigroup shifts.

\begin{defn}
Let $\Lambda$ be a two-sided shift space over an alphabet $G$. Consider the alphabet $\LAF^{1,1}$ and the full shift $\Sigma_{\LAF^{1,1}}^\Z$. Define \begin{equation}\label{forbiddenblocks_barLambda}\mathbf{F}:=\{\mathfrak{(F,G)}\in\LAF^{1,1}\times\LAF^{1,1}:
\FF_1(\Lambda,a)\neq \mathfrak{G}, \forall a\in\FF\}\end{equation}
and let the the follower-set shift of $\Lambda$ be defined as the 1-step shift $\bar \Lambda\subseteq\Sigma_{\LAF^{1,1}}^\Z$ given by $$\bar \Lambda:=X_\mathbf{F}.$$

In an analogous way we define the predecessor-set shift.
\end{defn}

\begin{prop}\label{row-finite_follower-set}
If $\Lambda\subseteq\Sigma_G^\Z$ is a row-finite shift then its follower-set shift is also row-finite.
\end{prop}

\begin{proof}
This follows directly from the fact that, for each $b\in \LA$, the follower set $\FF_1(\Lambda,b)$ is finite and hence $\FF_1(\Lambda,b)$ can only be followed in $\bar \Lambda$ by a finite number of elements of
$\LAF^{1,1}$.
\end{proof}

\begin{prop}\label{group_bar_Lambda}
Let $(\Lambda,\bullet)$ be a two-sided 1-block inverse semigroup shift over a group alphabet $(G,\cdot)$. Consider the group $(\LAF^{1,1},\cdot)$ defined in Corollary \ref{LAF_and_LAP}, the full inverse semigroup shift $(\Sigma_{\LAF^{1,1}}^\Z,\bullet)$
induced from $\cdot$ on $\LAF^{1,1}$ and let $\bar \Lambda$ be the follower-set shift of $\Lambda$. Then $(\bar \Lambda,\bullet)$ is a shift subsemigroup of $(\Sigma_{\LAF^{1,1}}^\Z,\bullet)$.
\end{prop}

\begin{proof}
Due to the infinite extension property, it is sufficient to show that
given $\hat x,\hat y\in\bar \Lambda^{\text{inf}}$, say $\hat x=(\hat x_i)_{i\in\Z}$ and $\hat y=(\hat y_i)_{i\in\Z}$, we have that $$\hat x\bullet \hat y^{-1}=(\hat x_i\cdot \hat y_i^{-1})_{i\in\Z}\in\bar
\Lambda^{\text{inf}}.$$

By the definition of $\bar \Lambda$, for each $i\in\Z$, there exist $a,c\in \LA$ such that $$\begin{array}{lcl} \hat x_i=\FF_1(\Lambda,a) & \text{and} &\hat x_{i+1}=\FF_1(\Lambda,b)\text{ for some } b\in
\FF_1(\Lambda,a),\\
\hat y_i=\FF_1(\Lambda,c) & \text{and} & \hat y_{i+1}=\FF_1(\Lambda,d)\text{ for some } d\in  \FF_1(\Lambda,c).
\end{array}$$

So, $\hat y_i^{-1}=\FF_1(\Lambda,c^{-1})$ and $\hat y_{i+1}^{-1}=\FF_1(\Lambda,d^{-1})$, with $d\in  \FF_1(\Lambda,c)$. Hence
$$ \hat x_i\cdot\hat y_i^{-1} =\FF_1(\Lambda,a)\cdot\FF_1(\Lambda,c^{-1})=\FF_1(\Lambda,a\cdot c^{-1}).
$$

Now, notice that $b\cdot d^{-1}\in \FF_1(\Lambda,a\cdot c^{-1})$ and is such that
$$\FF_1(\Lambda,b\cdot d^{-1})=\FF_1(\Lambda,b)\cdot \FF_1(\Lambda,d^{-1})=\hat x_{i+1}\cdot \hat y_{i+1}^{-1},$$
which implies that $\hat x\bullet \hat y^{-1}\in\bar \Lambda^{\text{inf}}$.

\end{proof}

\begin{defn}\label{ch5thetadef}
Let $(\Lambda,\bullet)\subset(\Sigma_G^\Z,\bullet)$ be a 1-block inverse semigroup shift and $(\bar \Lambda,\bullet)$ be the corresponding follower-set inverse semigroup shift. Define a 1-block code $\theta:\Lambda\to\Sigma_{\LAF^{1,1}}^\Z$, given, for all $x\in\Lambda$ and for all $i\in\Z$, by
$$\big(\theta(x)\big)_i:= \left\{\begin{array}{lcl} \o &,\ if&   x_i=\o\\\\
                                                    \FF_1(\Lambda,x_i) &,\ if&  x_i\neq\o.\end{array}\right.$$

\end{defn}

The following propositions characterize the image of $\theta$ and give conditions for it to be a topological conjugacy from $\Lambda$ to $\bar\Lambda$.\\

\begin{prop}\label{theta_properties} Let $(\Lambda,\bullet)\subset(\Sigma_G^\Z,\bullet)$ be a 1-block inverse semigroup shift, where $\Lambda$ is a Markov shift, and let $(\bar \Lambda,\bullet)$ be the corresponding follower-set inverse semigroup shift. If $\FF_1(\Lambda,1_G)\cap \Ppp_1(\Lambda,1_G)=\{1_G\}$, then
the map $\theta:\Lambda\to\Sigma_{\LAF^{1,1}}^\Z$ is such that:

\begin{enumerate}

\item $\theta:\Lambda\to\bar\Lambda$ is invertible and $\theta^{-1}$ is a 2-block code with memory 1 and anticipation 0, given, for all $\G\in\bar\Lambda$ and for all $i\in\Z$. by
\begin{equation}\label{theta_inverse}\big(\theta^{-1}(\G)\big)_i:= \left\{\begin{array}{lcl}  \o &,\ if&  \G_i=\o\\\\
                                                          a_i &,\ if&   \G_i\neq\o,\end{array}\right.\end{equation}
where $a_i$ is the unique element in $\G_{i-1}$ such that $\FF_1(\Lambda,a_i)=\G_i$, for all $i\leq l(\G)$, and $a_i=\o$ for all $i>l(\G)$.

\item $\theta$ is a topological conjugacy and an isomorphism.

\end{enumerate}

\end{prop}

\begin{proof}
{\color{white} .}

First we notice that for all $x\in\Lambda$ the sequence $\theta(x)$ is such that $l\big(\theta(x)\big)=l(x)$. Furthermore, let $\mathbf{F}$ be a set of forbidden words such that $\Lambda=X_{\mathbf{F}}$. We will break the first part of the proof into three steps:

\begin{description}

\item[Step 1:] $\theta:\Lambda\to\Sigma_{\LAF^{1,1}}^\Z$ is one-to-one.

Suppose that $x,y\in\Lambda$ are such that $\theta(x)=\theta(y)$. This means that for all $i\leq l(x)$ we have that $\FF_1(\Lambda,x_i)=\FF_1(\Lambda,y_i)$. Denote $\FF:=\FF_1(\Lambda,x_{i-1})=\FF_1(\Lambda,y_{i-1})$ and $\G:=\FF_1(\Lambda,x_i)=\FF_1(\Lambda,y_i)$. Then if $x_i,y_i\in \FF$ and, for any $z\in\G$, we have that $x_i,y_i\in\Ppp_1(\Lambda,z)=:\Ppp$, which implies that $x_i,y_i\in\FF\cap\Ppp$. But, as consequence of Proposition \ref{NormalSubgroup}, we have that $\FF_1(\Lambda,1_G)\cap \Ppp_1(\Lambda,1_G)$ is a normal subgroup of $\LA$ and therefore each nonempty set of the form $\FF\cap\Ppp$ has the same cardinality as $\FF_1(\Lambda,1_G)\cap \Ppp_1(\Lambda,1_G)$. Hence $x_i=y_i$ for all $i\leq l(x)$.

\item[Step 2:] $\theta(\Lambda)\subset\bar\Lambda$.

To prove this it is enough to show that $\theta(x)$ does not contain any forbidden block from \eqref{forbiddenblocks_barLambda}.

Indeed, given $x\in\Lambda$, for each $i\in\Z$ we have that $x_i\in \FF_1(\Lambda,x_{i-1})$ for all $i\leq l(x)$ and therefore $\FF_1(\Lambda,x_{i-1})\FF_1(\Lambda,x_i)\in \mathfrak{B}_2(\bar\Lambda)$ for all $i\leq l\big(\theta(x)\big)$. Thus, for all $x\in\Lambda^{\text{inf}}$, we have that $\bar x:=\theta(x)$ belongs to $\bar\Lambda$, while for $x\in\Lambda^{\text{fin}}$ we need additionally to check if $\bar x$ satisfies the infinite extension property in $\bar\Lambda$. We recall that due to the infinite extension property, we have that $\Lambda^{\text{fin}}$ is nonempty if and only if $\Lambda^{\text{inf}}$ has infinitely many elements (in particular $\O\in\Lambda^{\text{fin}}$ whenever $\Lambda^{\text{fin}}$ is nonempty). Hence, if $\O\in\Lambda^{\text{fin}}$ then, due to the injectivity of $\theta$, we get that $\theta(\Lambda^{\text{inf}})$ is a infinite subset of $\bar\Lambda^{\text{inf}}$, which means that $\theta(\O)=\O\in\bar\Lambda^{\text{fin}}$. Furthermore, if $x\in\Lambda^{\text{fin}}$ is not the empty sequence, then there exists $(y^n)_{n\in\N}$ such that $y^n\in\Lambda^{\text{inf}}$ for all $n$, $y^n_i=x_i$ for all $n\in\N$ and $i\leq l(x)$, and $y^m_{l(x)+1}\neq y^n_{l(x)+1}$ if $m\neq n$. Let $(\bar y^{n})_{n\in\N}$ be the sequence defined as $\bar y^n:=\theta(y^n)$ for each $n\in\N$. We have that each $y^n$ belongs to $\bar\Lambda^{\text{fin}}$. Now, since $\theta$ is a one-to-one sliding block code whose local rule has zero memory and anticipation, we have that $\bar y^n_i=\bar x_i$, for all $n\in\N$ and $i\leq l(\bar x)$. If by contradiction we suppose that $\bar y^m_{l(x)+1}= \bar y^n_{l(x)+1}$ for all except a finite number of indices $m\neq n$, then there exists an infinite set $\{y^{n_k}_{l(x)+1}\}_{k\in\N}\subset\FF_1(\Lambda,x_{l(x)})$ and $\mathfrak{J}\in\LAF^{1,1}$ such that $\FF_1(\Lambda,y^{n_k}_{l(x)+1})=\mathfrak{J}$ for all $k\in\N$. But this implies that taking $z\in\mathfrak{J}$ we have $\FF_1(\Lambda,x_{l(x)})\cap\Ppp(\Lambda,z)$ has infinite elements, a contradiction with that fact that $\FF\cap\Ppp$ is a coset of $\FF_1(\Lambda,1_G)\cap \Ppp_1(\Lambda,1_G)=\{1_G\}$. Hence, we have proved that $\bar x$ satisfies the infinite extension property in $\bar\Lambda$ and therefore $\bar x\in\bar\Lambda$.

\item[Step 3:] $\theta:\Lambda\to\bar\Lambda$ is onto.

Let $\G=(\G_i)_{i\in\Z}\in\bar\Lambda$. From the definition of $\bar\Lambda$, we have that for all $i\leq l(\G)$ there exists $a_i\in\G_{i-1}$ such that $\FF_1(\Lambda,a_i)=\G_i$. Using the same argument we used to prove that $\theta$ is one-to-one, we can prove that such element $a_i$ is unique. Thus, we can define $\Psi:\bar\Lambda\to\Sigma_A^\Z$ as the map given by the right side of \eqref{theta_inverse}. Note that $\Psi$ is length preserving and one-to-one. In fact, it is direct that $l(\G)=l\big(\Psi(\G)\big)$. Furthermore $\Psi(\G)=\O$ if and only if $\G=\O$, while if $\FF,\G\in\bar\Lambda$ are such that $\Psi(\FF)=\Psi(\G)=(a_i)_{i\in\Z}\neq\O$, then necessarily $l(\FF)=l(\G)$, $a_i\in\FF_{i-1}$ and $a_i\in\G_{i-1}$, and $\FF_1(\Lambda,a_i)=\FF_i$ and $\FF_1(\Lambda,a_i)=\G_i$ for all $i\leq l(\G)$.

    Now observe that $\O\in\bar\Lambda$ if and only if $\bar\Lambda$ has infinitely many elements (and infinitely many elements in $\bar\Lambda$ implies that $\Lambda$ also has infinitely many elements). Thus, if $\O\in\bar\Lambda^{\text{fin}}$ then $\O\in\Lambda^{\text{fin}}$, and therefore $\Psi(\O)=\O\in\Lambda$.
    On the other hand, suppose that $\G\in\bar\Lambda\setminus\{\O\}$. Then $x=\Psi(\G)$ is such that $x_i=a_i$ for each $i\leq l(x)=l(\G)$, where $a_{i-1}\in\G_{i-2}$ and $\FF_1(\Lambda,a_{i-1})=\G_{i-1}$, while $a_i\in\G_{i-1}=\FF_1(\Lambda,a_{i-1})$ and $\FF_1(\Lambda,a_i)=\G_i$. Thus, since $\Lambda$ is a Markovian shift and $a_i\in\FF_1(\Lambda,a_{i-1})$ for each $i\leq l(x)$, we have that $x$ does not contain any forbidden word of $\mathbf{F}$. Hence, if $l(\G)=\infty$, then $x\in\Lambda$. To conclude that $x=\Psi(\G)$ belongs to $\Lambda$ when $l(\G)<\infty$ we need to check that $x$ satisfies the infinite extension property. This can be made by using an argument analogous to that used to prove that $\theta(y)$ satisfies the infinite extension property for a finite sequence $y\in\Lambda$.

\end{description}

    Finally, notice that straightforwardly from the definition of $\Psi$ we have that $\Psi=\theta^{-1}$.

 To prove the second part of the proposition, notice that since $\theta$ is an invertible sliding block code between $\Lambda$ and $\bar\Lambda$, it is a conjugacy. To prove that it is a topological conjugacy, we just need to prove that $\theta$ is continuous or, equivalently, to prove that $\Psi=\theta^{-1}$ is continuous.
    Consider the local rule of $\Psi$, $$\big(\Psi(\G)\big)_n=\sum_{a\in \LA\cup\{\o\}}a\mathbf{1}_{C_a}\circ\sigma^{n}(\G).$$
    To prove that $\Psi$ is continuous, by Theorem 3.13 of \cite{GoncalvesSobottkaStarling2015_2}, it is enough to show that each set $C_a$ is a pseudo cylinder. From the definition of $\Psi$ we have that $\big(\Psi(\G)\big)_n=a$, if $a\in\G_{n-1}$, and $\FF_1(\Lambda,a)=\G_n$. Denote $\mathfrak{R}:=\G_{n}$ and $\mathfrak{T}:=\G_n$. Since,  $\LAF^{1,1}$ is a family of disjoint sets, then $\mathfrak{R}$ is the unique set of $\LAF^{1,1}$ which contains $a$ and therefore $C_a$ is the pseudo-cylinder $[\mathfrak{R}\ \mathfrak{T}]_{-1}^0$.

    Now, let $x,y\in\Lambda$, $\FF:=\theta(x)=\big(\FF_1(\Lambda,x_i)\big)_{i\in\Z}$ and $\G:=\theta(y)=\big(\FF_1(\Lambda,y_i)\big)_{i\in\Z}$. We have that
    $$\begin{array}{lcl}\theta(x\bullet y)&=& \theta\big((x_i\cdot y_i)_{i\in\Z}\big)\\\\
    &=&\big(\FF_1(\Lambda,x_i\cdot y_i)\big)_{i\in\Z}\\\\
    &=&\big(\FF_1(\Lambda,x_i)\cdot\FF_1(\Lambda,y_i)\big)_{i\in\Z}\\\\
    &=&\big(\FF_1(\Lambda,x_i)\big)_{i\in\Z}\bullet\big(\FF_1(\Lambda,y_i)\big)_{i\in\Z}
    =\theta(x)\bullet\theta(y),\end{array}$$
    which shows that $\Psi$ is an isomorphism between $(\Lambda,\bullet)$ and $(\bar\Lambda,\bullet)$.

%
%
%
%

\end{proof}

\begin{rmk}  Note that we could not assure that $\bar\Lambda\subset\theta(\Lambda)$ without the assumption that $\Lambda$ is a Markov shift. For instance, if we consider the inverse semigroup shifts in examples \ref{Example_F_infinite1} and \ref{Example_F_infinite1}, then $|\FF(\Lambda,1_{\LA})|=\LA$ and then $\bar\Lambda$ contains only the constant sequence $(\FF_i)_{i\in\Z}$, with $\FF_i=\FF(\Lambda,1_{\LA})$. However $\theta(\Lambda)$ contains the sequence $(\FF_i)_{i\in\Z}$ and also contains finite sequences $\G=(\G_i)_{i\in\Z}$ of all lengths such that  $\G_i=\FF(\Lambda,1_{\LA})$ for all $i\leq l(\G)$.
\end{rmk}

Now, let \begin{equation}\label{H_defn}\h:=\FF_1(\Lambda,1_G)\cap\Ppp_1(\Lambda,1_G),\end{equation} which is always non empty since $1_G\in\h$. Notice that due to Proposition \ref{NormalSubgroup} $\h$ is a
normal subgroup of $\LA$. Let
\begin{equation}\label{LAH}\LAH:=\LA/_{\h}=\{a\cdot\h:\ a\in \LA\},\end{equation} which is also a group with the operation (which we will also denote as $\cdot$) induced by the operation on $\LA$.

\begin{defn} Given a Markovian 1-block inverse semigroup shift $\Lambda$, we define $\bar\Lambda^{[0]}=\Lambda$ and, for $n\geq 1$, we define $\bar\Lambda^{[n]}$ as the follower-set shift of $\bar\Lambda^{[n-1]}$. Denote as
$\h^{[n]}:=\FF_1(\bar\Lambda^{[n]},1^{[n]})\cap \Ppp_1(\bar\Lambda^{[n]},1^{[n]})$, which is a normal subgroup of the group alphabet $\LA^{[n]}:=\mathfrak{B}_1(\bar\Lambda^{[n]})$.
\end{defn}

\begin{defn}\label{defn_fractal}
A Markovian 1-block inverse semigroup shift $\Lambda$ will be said to be {\em fractal shift} if, for all $n\geq 0$, we have that $\h^{[n]}$ is a singleton. In the particular case that $\bar\Lambda^{[1]}=\Lambda$ we will say that
$\Lambda$ is {\em self similar} and, if $\bar\Lambda^{[n]}=\bar\Lambda^{[n-1]}$ for some $n\geq 2$, we will say that $\Lambda$ is {\em self similar at level $n$}.
\end{defn}

We note that fractal shift spaces are exactly those for which we can apply the map $\theta$ infinitely many consecutive times, getting always a shift which is conjugate to $\Lambda$.

\begin{ex}
Let $G$ be a finite group, let $\Lambda\subset\Sigma_G^\Z$ be a Markovian 1-block inverse semigroup shift and suppose that $\Lambda$ is fractal. Then $\Lambda$ contains only constant sequences over the alphabet group
$\LA$.

In fact, since the alphabet is finite, we have that the numbers $a_n:=|\bar\Lambda^{[n]}|$, $n\geq 0$, form a non-increasing sequence. Therefore, $\Lambda$ shall be self similar at level $n$ for some $n$. But this implies
that $a_{n+1}=a_n$, which occurs if, and only if, $\FF_1(\bar\Lambda^{[n]},1^{[n]})$ is a singleton. Hence, since $\bar\Lambda^{[n]}$ is a group and $\FF_1(\bar\Lambda^{[n]},1^{[n]})$ is a singleton, we have that
$\bar\Lambda^{[n]}$ contains only constant sequences, and therefore (going backward from $\bar\Lambda^{[n]}$ to $\Lambda$) we get that $\Lambda$ contains only constant sequences over the subgroup $\LA\subset G$.

In particular, the above means that fractal shifts over finite alphabets are always self similar.
\end{ex}

\begin{ex}
The Markovian 1-block inverse semigroup shift given in Example \ref{ex_fractal} is a self similar fractal shift over an infinite alphabet.
\end{ex}

Now, let $S:\LAH\to\LA$ be an arbitrary section of
$\LAH$, i.e., an arbitrary map such that,
for all $\h_1\in\LAH$, $S(\h_1)\in\h_1$. It follows that:

\begin{prop} {\color{white}.}
\begin{enumerate}
\item For all $a\in\LA$, $(S(a\cdot\h)^{-1}\cdot a)\in\h$.

\item The map $\varphi:\LA\to\LAH\times\h$ given
 by $\varphi(a)=\bigl(a\cdot\h,S(a\cdot\h)^{-1}\cdot a\bigr)$ is a
 bijection. Moreover, $\varphi^{-1}:\LAH\times\h\to\LA$ is given
 by $\varphi^{-1}(a\cdot\h,h)=g$, where $g\in\LA$ is the unique
 element such that $S(a\cdot\h)^{-1}\cdot g=h$.

\item The group $(\LA,\bullet)$ is isomorphic to
 $(\LAH\times\h,\diamond)$ through the map $\varphi$, where the operation $\diamond$ is given by $$(\h_1,h_1)\diamond(\h_2,h_2):=\varphi\bigl[\varphi^{-1}(\h_1,h_1)\cdot\varphi^{-1}(\h_2,h_2)\bigr].$$
\end{enumerate}
\end{prop}

\begin{proof}

The proof is similar to the proof of Proposition 4.17 in \cite{sobottka2007}.

\end{proof}

We can alternatively write
$$(\h_1,h_1)\diamond(\h_2,h_2)=\bigl(\h_1\cdot\h_2,
 S(\h_1\cdot\h_2)^{-1}\cdot(g_1\cdot g_2)\bigr),$$
where $g_1,g_2\in\LA$ are the unique elements which satisfy
$$S(\h_1)^{-1}\bullet g_1=h_1,$$
$$S(\h_2)^{-1}\bullet g_2=h_2.$$

Note that on the first coordinate, the operation $\diamond$ coincides with the operation
$\cdot$ on $\LAH$.

\begin{defn} Define the two-sided Markov shift $\hat\Lambda\subset\Sigma_{\LAH}^\Z$, as the shift such that the infinite sequences are given by transitions:

$$\h_1\to\h_2\Longleftrightarrow\h_2\subseteq\F_1(\Lambda,a)\quad\text{ for\ all\ }a\in\h_1 .$$
\end{defn}

Note that the transitions above are well defined (independently of $a\in\h_1$ chosen). In fact, $\h_1=a'\cdot\h=a'\cdot\F_1(\Lambda,1_G)\cap\Pp_1(\Lambda,1_G)=\F_1(\Lambda,b)\cap\Pp_1(\Lambda,c)$ for some $b,c\in\LA$, and
therefore all the elements of $\h_1$ belong to $\Pp_1(\Lambda,c)$, which means that all of them have the same follower set.

Let $\star$ be the 1-block inverse semigroup operation on $\hat\Lambda\boxtimes\Sigma_\h^\Z$ induced by $\diamond$.
Recall that $\h$ is the identity of the group $\hat\Lambda$.

\begin{prop} $\hat\Lambda$ has the property that $\FF_1(\hat\Lambda,\h)\cap\Ppp_1(\hat\Lambda,\h)=\{\h\}$.
\end{prop}

\begin{proof}
Since $1_G\in\h=\FF_1(\Lambda,1_G)\cap\Ppp_1(\Lambda,1_G)$, then the transition $\h\to\h$ is allowed in $\hat\Lambda$ and therefore $\h\in \FF_1(\hat\Lambda,\h)\cap\Ppp_1(\hat\Lambda,\h)$. Now, let $\mathcal{J}\in \FF_1(\hat\Lambda,\h)\cap\Ppp_1(\hat\Lambda,\h)$. Then the transitions $\h\to\mathcal{J}\to\h$ are allowed in $\hat\Lambda$, which means that $\mathcal{J}\subset\FF_1(\Lambda,1_G)$ and, for any $a\in \mathcal{J}$, we have that $\h\subset \FF_1(\Lambda,a)$, which in turn means that $a\in\Ppp_1(\Lambda,1_G)$. In other words, if $a\in\mathcal{J}$ then $a\in\FF_1(\Lambda,1_G)\cap\Ppp_1(\Lambda,1_G)=\h$, that is $\mathcal{J}\subset\h$. Since the sets of $\hat\Lambda$ are cosets of $\h$, it follows that $\mathcal{J}=\h$.

\end{proof}

\begin{defn} Let $(\Lambda,\bullet)\subset(\Sigma_G^\Z,\bullet)$ be a 1-block inverse semigroup shift. Define the 1-block code $\phi:\Lambda\to\Sigma_{L_{\hat\Lambda}}^\Z\boxtimes\Sigma_\h^\Z$, for all $x\in\Lambda$ and for all
$i\in\Z$, by

\begin{equation}\bigl(\phi(x)\bigr)_i:=\left\{\begin{array}{ll}\o  & \text{if } x_i=\o\\\\
                                                                \varphi(x_i)=\big(x_i\cdot\h,S(x_i\cdot\h)^{-1}\cdot x_i\big)& \text{if } x_i\neq\o.\end{array}\right.\end{equation}

\end{defn}

\begin{prop}\label{phi_properties}
Let $(\Lambda,\bullet)$ be a Markovian 1-block inverse semigroup shift. Then  $\phi:\Lambda\to\hat\Lambda\boxtimes\Sigma_\h^\Z$ is a topological conjugacy and an isomorphism between $(\Lambda,\bullet)$ and $(\hat\Lambda\boxtimes\Sigma_\h^\Z,\star)$. In particular, given $\mathbf y =(\mathbf y_i)_{i\in\Z}\in \hat\Lambda\boxtimes\Sigma_\h^\Z$, we have that \begin{equation}\label{phi_inverse}\bigl(\phi^{-1}(\mathbf y)\bigr)_i:=\left\{\begin{array}{ll}\o  & \text{if } \mathbf y_i=\o\\\\
                                                                \varphi^{-1}(\mathbf y_i)& \text{if } \mathbf y_i\neq\o.\end{array}\right.\end{equation}
\end{prop}

Since the proof of Proposition \ref{phi_properties} follows the same outline of the proof of Proposition \ref{theta_properties}, we will just give a sketch of its proof.

\begin{proof}[Sketch of the proof of Proposition \ref{phi_properties}]{\color{white}.}

\begin{enumerate}[label=\arabic*.]

\item Note that for $\phi:\Lambda\to\Sigma_{L_{\hat\Lambda}}^\Z\boxtimes\Sigma_\h^\Z$ we have that for all $x=(x_i)_{i\in\Z}$, if $\mathbf y:=\phi(x)$, then $l(\mathbf y)=l(x)$;

\item Observe that $\mathbf y$ does not contain forbidden words of $\hat\Lambda\boxtimes\Sigma_\h^\Z$, which implies that $\mathbf y\in \hat\Lambda\boxtimes\Sigma_\h^\Z$ whenever $l(\mathbf y)=\infty$;

\item Note that since $\phi$ is a sliding block code whose local rule is 1-block and a bijection between $\LA$ and $L_{\hat\Lambda}$, then $\phi$ is one-to-one;

\item Use the same argument used to proof of Proposition \ref{theta_properties} to check that if $l(\mathbf y)<\infty$ then $\mathbf y=\phi(x)$ satisfies the infinite extension property in $\hat\Lambda\boxtimes\Sigma_\h^\Z$. Thus, we get $\phi(\Lambda)\subset \hat\Lambda\boxtimes\Sigma_\h^\Z$.

\item Define $\Psi:\hat\Lambda\boxtimes\Sigma_\h^\Z\to\Sigma_G^\Z$ and show that, for all $\mathbf y\in \hat\Lambda\boxtimes\Sigma_\h^\Z$, $\Psi(\mathbf y)$ does not contain any forbidden word of $\Lambda$;

\item Use that $\Psi$ is one-to-one to get that $\Psi(\mathbf y)\in\Lambda$, for all $\mathbf y\in \hat\Lambda\boxtimes\Sigma_\h^\Z$.

\item Observe that $\Psi\big(\phi(x)\big)=x$ and thus we have that $\phi:\Lambda\to\hat\Lambda\boxtimes\Sigma_\h^\Z$ is a bijection and $\phi^{-1}=\Psi$;

\item Directly from the definitions of $\star$ and $\phi$ we get that $\phi$ is an isomorphism between $(\Lambda,\bullet)$ and $(\hat\Lambda\boxtimes\Sigma_\h^\Z,\star)$;

\item Note that $\phi:\Lambda\to\hat\Lambda\boxtimes\Sigma_\h^\Z$ satisfies the hypotheses of \cite[Theorem 3.13]{GoncalvesSobottkaStarling2015_2} to conclude that $\phi$ is continuous and then it is a topological conjugacy.

\end{enumerate}

\end{proof}

\begin{theo}\label{lastsectionmaintheo}
If $(\Lambda,\bullet)$ is a two-sided Markovian 1-block inverse semigroup shift, then it is isomorphic and topologically conjugate, via a sliding block code with zero memory and anticipation, to a two-sided Markovian inverse
semigroup shift $(\mathbb{F}\boxtimes \Sigma_B^\Z,\star)$, where $\mathbb{F}$ is a fractal shift, $\Sigma_B^\Z$ is a full shift over some alphabet $B$ (finite or infinite) and $\star$ is a block operation with anticipation 0 and memory $k$, for some
$k\geq 0$.
\end{theo}

\begin{proof}
The proof of this result follows the same outline of Theorem 1 in \cite{kitchens}, that is, starting with $\Lambda$ we shall apply the maps $\phi$ and $\theta$ alternately and recursively, therefore obtaining the product shift
of a shift space $\mathbb{F}$ with a full shift $\Sigma_B^\Z$:

\small{
$$\xymatrix{
{ \Lambda} \ar[r]|-\phi  & { \hat{\Lambda}\boxtimes\h_1^\Z}\ar[dd]|-{\theta\times id}\\\\
                         & { \SgHF\boxtimes\Sigma_{\h_1}^\Z}\ar[rr]|-{\phi\times id}        & & {\SgHFH\boxtimes\Sigma_{\h_2}^\Z\boxtimes\Sigma_{\h_1}^\Z} \ar[dd]|-{\theta\times id\times id} \\\\
                         &                                                               & & {\SgHFHF\boxtimes\Sigma_{\h_2}^\Z\boxtimes\Sigma_{\h_1}^\Z} \ar[rr]|-{\phi\times id\times id} &&
                         {\SgHFHFH\boxtimes\Sigma_{\h_3}^\Z\boxtimes\Sigma_{\h_2}^\Z\boxtimes\Sigma_{\h_1}^\Z }\ar@{-->}[dd]|-{\theta\times id\times id\times id}\\\\
                         &                                                               & &                                                                                           &&  {\mathbb{F}\boxtimes \Sigma_B^\Z}\\
                         }
$$}\normalsize
\vspace{.5cm}

In particular, $\Lambda$ and $\mathbb{F}\boxtimes \Sigma_B^\Z$ are topologically conjugate and we can define an operation $\star$ on $\mathbb{F}\boxtimes \Sigma_B^\Z$ which corresponds to the projection of $\bullet$ on
$\mathbb{F}\boxtimes \Sigma_B^\Z$ (or equivalently, it can be constructed using the operations which are defined in Propositions \ref{group_bar_Lambda} and \ref{phi_properties} at each step of the above procedure). Note
that, since $\theta^{-1}$ is a 2-block code with memory 1 and anticipation 0, $\star$ will have anticipation 0 and will have memory $k$ (where $k$ is less than or equal to the number of times we applied $\theta$ in the procedure).

Thus, we only need to prove that the above procedure will result in $\mathbb{F}$ being a fractal shift after some finite number of steps. Suppose the contrary, that is, suppose that there is not a finite number of steps after which
the previous procedure results in $\mathbb{F}$ being a fractal shift. Let $\{Q_n\}_{n\geq 1}$ be a family of shift spaces, where $Q_n$ is the shift space obtained by applying $n$ times $\phi$ and $\theta$ alternately on
$\Lambda$, that is, $$Q_n:=\mathbb{F}_n\boxtimes\Sigma_{\h_n}^\Z\cdots\boxtimes\cdots\Sigma_{\h_2}^\Z\boxtimes\Sigma_{\h_1}^\Z,$$
where $\mathbb{F}_n$ is some non-fractal shift space. Note that $\Lambda$ is conjugate to $Q_n$ and, since the sequence of maps which take sequences in $\Lambda$ to sequences in $Q_n$ is composed of maps that act as 1-block
codes, we have that any element $a\in\LA$ can be represented as $(f,y_n,y_{n-1},\ldots,y_1)\in L_{Q_n}$. Moreover, any point $(f,y_n,y_{n-1},\ldots,y_1)\in L_{Q_n}$ is the
image of at least one element of $\LA$. In other words, for all $n$ there is an onto map $\alpha_n:\LA\to L_{\Sigma_{\h_n}^\Z}\times\cdots\times L_{\Sigma_{\h_2}^\Z}\times L_{\Sigma_{\h_1}^\Z}$ such that
$$\alpha_n(a)=(y_n,\ldots,y_2,y_1)\Longleftrightarrow a \text{ is represented in }L_{Q_n}\text{ as } (f,y_n,\ldots,y_2,y_1),\text{ for some } f\in \mathbb{F}_n.$$
Now define an onto map $\alpha:\LA\to \prod_{i\in\N} L_{\Sigma_{\h_i}^\Z}$ by setting $\alpha(a)=(y_i)_{i\in\Z}$ such that, for all $n\in\N$, we have $(y_n,\ldots,y_1)=\alpha_n(a)$. If we never reach a fractal shift, then
infinitely many alphabets $L_{\Sigma_{\h_i}^\Z}$ have more than 1 element and, therefore, $\prod_{i\in\N} L_{\Sigma_{\h_i}^\Z}$ is uncountable. This would imply that $\LA$ is also uncountable, a contradiction.

\end{proof}

\subsection{Results for one-sided shift spaces}

We remark that although the previous results in this section were proved just for the case when $\Lambda$ is two-sided, such assumption is not a constraint to characterize one-sided 1-block inverse semigroup shifts. In fact, if
$\Lambda'\subseteq\Sigma_G^\N$ is a one-sided shift space, since we are assuming that $\s(\Lambda')=\Lambda'$, then we can always take its inverse limit, which will be a two-sided shift space $\Lambda\subseteq\Sigma_G^\Z$ such that
$\pi(\Lambda)=\Lambda'$ (see \cite[Remark 2.6]{GoncalvesSobottkaStarling2015_2}). Thus, we can use \cite[Proposition 4.6]{GoncalvesSobottkaStarling2015_2} to characterize $\Lambda'$ by characterizing its inverse limit.

On the other hand, we remark that we have only considered two-sided shift spaces because higher block codes are in general non-invertible for one-sided shifts, see Corollary 3.22 in \cite{GoncalvesSobottkaStarling2015}, and because for one-sided shifts the map $\theta$  is also non-invertible.
However, since both the higher block code and the map $\theta$ fail to be invertible just on the finite sequences of one-sided shift spaces, if we consider the set $A^\NZ$ with the product topology, instead of considering it with the topologies proposed in \cite{Ott_et_Al2014} and \cite{GoncalvesSobottkaStarling2015_2}, then it is possible to use the results of Section \ref{Isomorphism of two-sided Markovian} to prove a generalization of Kitchens' result: Any $M$-step subshift $\Lambda\subset A^\NZ$ with a 1-block group operation induced from a group operation on $A$ is topologically conjugate to the Cartesian product of a full shift with a fractal shift.

\section*{Acknowledgments}

\noindent D. Gon\c{c}alves was partially supported by CNPq and Capes project PVE085/2012.

\noindent M. Sobottka was supported by CNPq-Brazil grants 304813/2012-5, 480314/2013-6 and 308575/2015-6. Part of this work was carried out while the author was postdoctoral fellow of CAPES-Brazil at Center for Mathematical Modeling, University of Chile.

\noindent C. Starling was supported by CNPq, and work on this paper occured while the author held a postdoctoral fellowship at UFSC.



\begin{thebibliography}{20}





\bibitem{Clifford1967} {\sc Clifford, A.~H.} {\sc and}  {\sc Preston, G.~B.} (1967).
\newblock ``The algebraic theory of semigroups II'',
\newblock Mathematical Surveys and Monographs,~\textbf{7}.










\bibitem{GoncalvesSobottkaStarling2015}
{\sc Gon\c{c}alves, D., Sobottka, M.} {\sc and} {\sc Starling, C.}  (2016).
\newblock  {\em Sliding block codes between shift spaces over infinite alphabets},
\newblock  Math. Nachr., \textbf{289}, 17-18, 2178--2191.

\bibitem{GoncalvesSobottkaStarling2015_2}
{\sc Gon\c{c}alves, D., Sobottka, M.} {\sc and} {\sc Starling, C.}  (2017).
\newblock  {\em Two-sided shift spaces over infinite alphabets},
\newblock  accepted in Journal of the Australian Mathematical Society.


\bibitem{kitchens}
{\sc Kitchens, B.~P.} (1987).
\newblock {\em Expansive dynamics on zero-dimensional groups},
\newblock  Ergodic Theory and Dynamical Systems, ~\textbf{7},~2, 249--261.

\bibitem{La98}
{\sc M.V. Lawson.} (1998).
\newblock {\em Inverse Semigroups: The Theory of Partial Symmetries}.
\newblock World Scientific.

\bibitem{LindMarcus}
{\sc Lind, D.~A.} {\sc and} {\sc Marcus, B.} (1995).
\newblock ``An introduction to symbolic dynamics and coding'',
\newblock {\em Cambridge, Cambridge University Press.}


\bibitem{marcus}
{\sc Sindhushayana, N.~T.}, {\sc Marcus, B.} {\sc and} {\sc Trott, M.} (1997).
\newblock  {\em Homogeneous shifts},  IMA J. Math. Control Inform.,
~\textbf{14}, ~3, 255--287

\bibitem{Ott_et_Al2014}
{\sc Ott, W.}, {\sc Tomforde, M.} {\sc and} {\sc WILLIS,
P.~N.} (2014).
\newblock  {\em One-sided shift spaces over infinite alphabets},
\newblock New York Journal of Mathematics. NYJM Monographs 5. State University of New York, University at Albany, Albany, NY. 54 pp.


\bibitem{sobottka2007}
{\sc Sobottka, M.} (2007)
\newblock {\em Topological quasi-group shifts},
\newblock Disc. and Contin. Dyn. Syst.,~\textbf{17}, 77--93.

\bibitem{SG}
{\sc Sobottka, M.} {\sc and} {\sc Gon\c{c}alves, D.}  (2017).
\newblock  {\em A note on the definition of sliding block codes and the Curtis-Hedlund-Lyndon Theorem},
\newblock accepted in Journal of Cellular Automata.




\end{thebibliography}
\end{document}